\documentclass[11pt, a4paper, reqno, twoside]{amsart}
\usepackage{amsmath,amssymb,amsthm}

\newtheorem{lemma}{Lemma}[section]
\newtheorem{theorem}{Theorem}[section]

\newtheorem{rem}[theorem]{Remark}
\newtheorem{definition}{Definition}[section]
\newtheorem{claim}{Claim}[section]

\renewcommand{\thelemma}{\thesection.\arabic{lemma}}

\numberwithin{equation}{section}

\newcommand{\ve}{\varepsilon}
\newcommand{\R}{\mathbb R}
\newcommand{\N}{\mathbb N}

\newcommand{\jb}[1]{\left\langle #1 \right\rangle}

\newcommand{\pa}{\partial}
\newcommand{\en}{{\mathcal E}}
\newcommand{\Norm}[1]{|\!|\!|#1|\!|\!|}
\DeclareMathOperator{\supp}{\rm supp}
\title[
 The exterior Neumann problems
 of semilinear wave equations in $2$D] {Almost global existence
for exterior Neumann problems of semilinear wave equations in $2$D}

\author[S. Katayama]{Soichiro Katayama}
\address{Soichiro Katayama\\ Department of Mathematics,
Wakayama University, 930 Sakaedani,
Wakayama 640-8510, Japan
Tel: +81-73-457-7343}
\email{katayama@center.wakayama-u.ac.jp}

\author[H. Kubo]{Hideo Kubo}
\address{Hideo Kubo\\ Division of Mathematics,
Graduate School of Information Sciences,
Tohoku University, Sendai 980-8579, Japan
Tel: +81-22-795-4628}
\email{kubo@math.is.tohoku.ac.jp}

\author[S. Lucente]{Sandra Lucente}
\address{Sandra Lucente \\ Dipartimento di Matematica,
Universit\`a degli Studi di Bari,  Via Orabona 4, 70125 Bari,
Italy. Tel. +39-080-5442275} \email{lucente@dm.uniba.it}

\begin{document}
\baselineskip16pt
\begin{abstract}
The aim of this article is to prove  an \lq\lq almost\rq\rq \, global
existence result for some semilinear wave equations in the plane outside a bounded
convex obstacle with the Neumann boundary condition.
\end{abstract}

\maketitle

\thispagestyle{empty}

%\tableofcontents

\section{Introduction}
Let ${\mathcal O}$ be an open
bounded convex domain with smooth boundary
in ${\R}^2$ and put $\Omega:={\R}^2 \setminus
\overline{\mathcal O}$.
Let $\partial_\nu$ denote the outer normal derivative on $\partial \Omega$.

We consider the mixed problem for semilinear wave equations in $\Omega$ with the Neumann boundary condition:
\begin{equation}\label{eq.PMN}
\begin{array}{ll}
 (\partial_t^2-\Delta) u =G(\partial_t u, \nabla_x u), & (t,x) \in (0,\infty)\times \Omega,\\
 \partial_\nu u(t,x)=0, &(t,x) \in (0,\infty)\times \partial\Omega,\\
 u(0,x)=\phi(x), &x\in \Omega,\\
 \partial_t u(0,x)=\psi(x), & x\in \Omega,\\
\end{array}
\end{equation}
where $\phi$ and $\psi$
are  ${\mathcal C}^\infty$-functions
compactly supported in $\overline\Omega$,
and $G: \R^3\to \R$ is a nonlinear function.
We will study the case of the cubic nonlinearity with small initial data and obtain an estimate from below
for the lifespan of the solution in terms of the size of the initial data.
Here by the expression \lq\lq small initial data'' we mean that there exist $m\in \N$, $s\in \R$ and a small number $\ve>0$ such that
\begin{equation*}
\|\phi\|_{H^{m+1,s}(\Omega)}+\|\psi\|_{H^{m,s}(\Omega)}\le \ve,
\end{equation*}
where the weighted Sobolev space $H^{m,s}(\Omega)$ is endowed with the norm
\begin{equation}\label{eq.datanorm}
\|\varphi\|_{H^{m,s}(\Omega)}^2:=\sum_{|\alpha|\le m}\int_{\Omega}
 (1+|x|^2)^s
|\partial_x^\alpha \varphi(x)|^2 d x.
\end{equation}
A large amount of works has been devoted to the study of the mixed problem for nonlinear wave equations in an exterior domain $\Omega \subset \R^n$ for $n\ge 3$, mostly with the Dirichlet
boundary condition.
To our knowledge very few results deal with the global existence
or the lifespan estimate for the exterior mixed problems of nonlinear wave equations in $2$D; in \cite{SSW} the global existence for the case of the Dirichlet boundary condition and the nonlinear terms depending only
on $u$ is considered;
in \cite{KP} one of the authors obtained an almost global
existence result for small initial data under the assumptions that $|G(\partial u)|\simeq (\partial u)^3$, the obstacle is star-shaped and
the boundary condition is of the Dirichlet type (see Remark~\ref{Rem1.4} below for the detail).

Here we will treat the problem with the Neumann boundary condition
in $2$D and obtain an analogous result to \cite{KP}. However, because we have a weaker decay property for the solution to the Neumann exterior problem of linear wave equations in $2$D
(see Secchi and Shibata \cite{SeSh03}), we will obtain
a slightly worse lifespan estimate than in the Dirichlet case.

For simplicity, we assume that
the nonlinear function $G$  in \eqref{eq.PMN} is a homogeneous polynomial of cubic order.
Equivalently, writing $\partial u=(\partial_tu, \nabla_x u)$, this means that
\begin{equation}\label{eq.semiG}
G(\partial u)=\sum_{0\le \alpha\le \beta\le \gamma\le 2} g_{\alpha,\beta,\gamma}(\partial_\alpha u)(\partial_\beta u)(\partial_\gamma u)
\end{equation}
with $g_{\alpha,\beta,\gamma}\in \R$ and $(\partial_0,\partial_1,\partial_2):=(\partial_t, \partial_{x_1},\partial_{x_2})$.

\smallskip
As usual, to consider smooth solutions to the mixed problem, we need some compatibility conditions (see \cite{KaKu08}).
Note that, for a nonnegative integer $k$ and a smooth function $u=u(t,x)$ on $[0,T)\times \Omega$, we have
\begin{equation}
\label{CC0}
\partial_t^k\left(G(\partial u)\right)=G^{(k)}[u, \partial_t u, \ldots,
\partial_t^{k+1}u],
\end{equation}
where for $\mathcal C^1$ functions  $(p_0, p_1, \ldots, p_{k+1})$ we put
\begin{align*}
G^{(k)}[p_0, p_1, \ldots, p_{k+1}]=&\sum_{k_1+k_2+k_3=k}
g_{0, 0, 0} p_{k_1+1}p_{k_2+1}p_{k_3+1}+
\sum_{k_1+k_2+k_3=k} \sum_{\gamma=1}^2
g_{0,0,\gamma}p_{k_1+1}p_{k_2+1}(\partial_\gamma p_{k_3})\\
&{}+\sum_{k_1+k_2+k_3=k} \sum_{1\le \beta\le \gamma\le 2}
g_{0,\beta,\gamma} p_{k_1+1}(\partial_\beta p_{k_2})(\partial_\gamma p_{k_3})\\
&{}+\sum_{k_1+k_2+k_3=k} \sum_{1\le \alpha\le \beta\le \gamma\le 2}
g_{\alpha,\beta,\gamma} (\partial_\alpha p_{k_1})(\partial_\beta p_{k_2})(\partial_\gamma p_{k_3}).
\end{align*}
\begin{definition}\label{CCN}
To the mixed problem \eqref{eq.PMN} we can associate the recurrence sequence $\{v_j\}_{j\in\N^*}$ with $v_j:\overline{\Omega}\to \R$ such that
$$
\begin{array}{l}
 v_0= \phi,\\
 v_1= \psi,\\
 v_j=\Delta v_{j-2}+G^{(j-2)}[v_0, v_1, \ldots, v_{j-1}], \quad j\ge 2,
\end{array}
$$
where $\N^*$ denotes the set of nonnegative integers
and $G^{(k)}$ is defined as above {\rm(}cf.~\eqref{CC0}{\rm)}.
We say that $(\phi, \psi, G)$ satisfies the compatibility
condition  of infinite order in
$\Omega$ for \eqref{eq.PMN} if $\phi,\psi\in {\mathcal C}^\infty(\overline{\Omega})$,
and one has
$$
\partial_\nu v_j(x)=0, \quad x\in \partial \Omega
$$
for all $j\in \N^*$.
\end{definition}

Our aim is to prove the following result.

\begin{theorem}\label{thm.mainsemi}
Let ${\mathcal O}$ be a convex obstacle. Consider the  semilinear mixed problem \eqref{eq.PMN} with given
compactly supported initial data
$(\phi,\psi)\in \mathcal C^\infty(\overline{\Omega})\times  {\mathcal C}^\infty(\overline{\Omega}) $ and a given
nonlinear term $G(\partial u)$ which is a
homogeneous polynomial of cubic order as in \eqref{eq.semiG}.
Assume that $(\phi, \psi,  G)$ satisfies the compatibility condition
 of infinite order in $\Omega$ for \eqref{eq.PMN}.

Under these assumptions, there exist $\varepsilon_0>0$, $m\in \N$, $s\in \R$ such that, if $\varepsilon \in (0,\varepsilon_0]$ and
\begin{equation}\label{eq.smalldata.semi}
\|\phi\|_{H^{m+1,s}(\Omega)}+\|\psi\|_{H^{m,s}(\Omega)} \le \ve,
\end{equation}
then the mixed problem
\eqref{eq.PMN} admits a unique solution $u \in
{\mathcal C}^\infty([0,T_\varepsilon)\times \Omega)$ with
\begin{equation}\label{eq.lifespanT1.semi}
T_\varepsilon \ge \exp(C \varepsilon^{-1}),
\end{equation}
where $C>0$ is a suitable constant which is uniform with respect to $\ve\in (0,\ve_0]$.
\end{theorem}
\begin{rem}
\normalfont
The only point where we require that the obstacle ${\mathcal O}$ is convex
is to gain the local energy decay (see Lemma~\ref{LocalEnergyDecay} below). In general one
 can  treat the obstacles for which Lemma~\ref{LocalEnergyDecay} holds.
Unfortunately, for the Neumann problems in $2$D, up to our knowledge it is not known if there exists non-convex obstacles satisfying such
a local energy decay.
\end{rem}

\begin{rem}
\normalfont
One can ask if it is possible to gain a global existence result maintaining our assumption on the growth of $G$.
In general the answer to this question is negative since
the blow-up in finite time occurs for $F=(\pa_t u)^3$ when $n=2$. Indeed, it was proved in \cite{God93} that for any $R>0$ we can
find initial data such that the blow-up for the corresponding Cauchy problem occurs in the region $|x|>t+R$.
This result shows the blow-up for the exterior problem with any
boundary condition if we choose sufficiently large $R$, because the solution in $|x|>t+R$ is not affected
by the obstacle and the boundary condition, thanks to the finite propagation property (see \cite{KaKu12} for the corresponding discussion
in $3$D).

In order to look for global solutions one could investigate the exterior problem with suitable nonlinearity satisfying the so-called \it null condition. \rm
\end{rem}

%%%%%%%%%%%%%%
\begin{rem}\label{Rem1.4}
\normalfont
If we consider the Cauchy problem in $\R^2$,
or the Dirichlet problem in a domain exterior to a star-shaped obstacle in $2$D,
an analogous result to Theorem~$\ref{thm.mainsemi}$ holds with
\begin{equation}
\label{sharplife}
T_\ve \ge \exp (C\varepsilon^{-2}),
\end{equation}
and this lifespan estimate is known to be sharp (see \cite{God93} for the Cauchy problem and \cite{KP} for the Dirichlet problem).
One loss of the logarithmic factor in the decay estimates causes
this difference between the lifespan estimates \eqref{eq.lifespanT1.semi} and
\eqref{sharplife} (see Theorem~\ref{main} and Remark~\ref{Rem83}
below).
It is an interesting problem whether our lower bound \eqref{eq.lifespanT1.semi}
is sharp or not for the Neumann problem.
\end{rem}

\section{Preliminaries}
In this section we introduce some notation which will be used throughout this paper
and some basic lemmas for the proof of Theorem \ref{thm.mainsemi}.

Throughout the paper we shall assume $0\in {\mathcal O}$ so that we have
$|x|\ge c_0$ for $x\in \Omega$ for some positive constant $c_0$.
We shall also assume that
$\overline{\mathcal O}\subset B_1$,
where
$B_r$ stands for an open ball
with radius $r$ centered at the origin of ${\R}^2$.
Thus a function $v=v(x)$ on $\Omega$ vanishing for $|x|\le 1$ can be naturally regarded
as a function on $\R^2$.

%%%%%%%%%%%%%%%%%%%%%%%%%%%%%%%%%%%%%%%%%%%
\subsection{Notation}
%%%%%%%%%%%%%%%%%%%%%%%%%%%%%%%%%%%%%%%%%%%
Let us start with some standard notation.
\begin{itemize}
\item We put $\langle y\rangle :=\sqrt{1+|y|^2}$ for $y\in \R^d$ with $d\in \N$.
\item Let $A=A(y)$ and $B=B(y)$
be two positive functions of some variable $y$, such as $y=(t,x)$ or $y=x$, on suitable domains.
We write $A\lesssim B$ if there exists a positive constant $C$ such that
$A(y)\le C B(y)$ for all $y$ in the intersection of the domains of $A$ and $B$.
\item The $L^2(\Omega)$ norm is denoted by $\|\cdot\|_{L^2_{\Omega}}$, while the norm $\| \cdot \|_{L^2}$ without any other index stands for $\|\cdot \|_{L^2(\R^2)}$.
Similar notation will be used for the $L^\infty$ norms.
\item For a time-space depending function $u$ satisfying
$u(t,\cdot)\in X$ for $0\le t< T$ with
a Banach space $X$,
we put $\|u\|_{L^\infty_TX}:=\sup_{0\le t< T}\|u(t,\cdot)\|_X$.
For the brevity of the description, we sometimes use the expression
$\|h(s,y)\|_{L^\infty_tL^\infty_\Omega}$ with dummy variables $(s,y)$
for a function $h$ on $[0,t)\times \Omega$,
which means
$\sup_{0\le s<t}\|h(s, \cdot)\|_{L^\infty_\Omega}$.
\item
For $m\in \N$ and $s\in \R$, by $H^{m,s}(\Omega)$ we denote the weighted Sobolev space with norm defined by \eqref{eq.datanorm}.
Moreover $H^m(\Omega)$ and $H^m(\R^2)$ are the
standard Sobolev spaces.
\item We denote by $\mathcal C_0^\infty(\overline\Omega)$ the set of smooth functions defined on $\overline{\Omega}$ which vanish outside $B_R$ for some $R>1$.
\end{itemize}

Let $\nu\in \R$. We put
\begin{equation*}
w_\nu(t,x)=\langle x \rangle^{-1/2} \langle t-|x|\rangle^{-\nu}+\langle t+|x|\rangle^{-1/2}\langle t-|x|\rangle^{-1/2}.
\end{equation*}
This weight function $w_\nu$ will be used repeatedly in
the \it a priori estimates of the solution $u$ to \eqref{eq.PMN}.
We shall often use the following inequality
\begin{equation}\label{eq.ome}
w_\nu(t,x)\lesssim \langle t+|x|\rangle^{-1/2}(\min\{\langle x\rangle, \langle t-|x| \rangle\})^{-1/2}, \quad \nu \ge 1/2.
\end{equation}

For $\nu$, $\kappa>0$ we put
$$
W_{\nu,\kappa}(t,x)=
\langle t+|x|\rangle^\nu \left (\min\{\langle x\rangle, \langle t-|x|\rangle\}\right)^\kappa.
$$
%}}

Finally, for $a\ge 1$ we set
$$
\Omega_a=\Omega \cap B_a.
$$
Since $\overline{\mathcal O}\subset B_1$, we see that $\Omega_a\not= \emptyset $ for any $a\ge 1$.
%%%%%%%%%%%%%%
\subsection{Vector fields associated with the wave operator}
We introduce the vector fields\,:
$$
\Gamma_0:=\partial_0=\partial_t, \quad \Gamma_1:=\partial_1=\partial_{x_1},\quad \Gamma_2:=\partial_2=\partial_{x_2}, \quad
\Gamma_3:=\Lambda: =x_1 \partial_2-x_2\partial_1.
$$
Denoting $[A,B]:=AB-BA$, we have
\begin{equation}\label{eq.commute}
[\Gamma_i, \partial_{t}^2-\Delta]=0, \quad i=0,\dots,3,
\end{equation}
and also
$$
\begin{array}{ll}
[\Gamma_i,\Gamma_j]=0, & i,j=0,1,2,\cr
[\Gamma_0,\Gamma_3]=0, &\ \cr
[\Gamma_1,\Gamma_3]=\Gamma_2, &\cr
[\Gamma_2,\Gamma_3]=-\Gamma_1.
\end{array}
$$
Hence, for $i,j=0,1,2,3$,
we have
$[\Gamma_i, \Gamma_j]=\sum_{k=0}^3 c_{ij}^k\,\Gamma_k$
with suitable constants $c_{ij}^k$.
Moreover, for $i=0,1,2$ and $j=0,1,2,3$ we also have
$[\partial_i, \Gamma_j]=\sum_{k=1}^2 d_{ij}^k \partial_k$
with suitable constants $d_{ij}^k$.

We put $\partial=(\partial_0,\partial_1, \partial_2)$, $\partial_x=(\partial_1, \partial_2)$,
$\Gamma=(\Gamma_0, \Gamma_1, \Gamma_2, \Gamma_3)=(\partial, \Lambda)$
and $\widetilde{\Gamma}=(\Gamma_1, \Gamma_2, \Gamma_3)=(\partial_x, \Lambda)=
%{\blue{
(\nabla_x, \Lambda)
%}}
$.
The standard multi-index notation will be used for these sets of vector fields,
such as $\pa^\alpha=\pa_0^{\alpha_0}\pa_1^{\alpha_1}\pa_2^{\alpha_2}$
with $\alpha=(\alpha_0,\alpha_1,\alpha_2)$ and $\Gamma^\gamma=\Gamma_0^{\gamma_0} \cdots \Gamma_3^{\gamma_3}$ with $\gamma=(\gamma_0, \dots, \gamma_3)$.

For $\rho\ge 0$, $k\in \N$ and
functions $v_0=v_0(x)$ and $v_1=v_1(x)$, we put
\begin{eqnarray*}
\mathcal A_{\rho, k}[v_0,v_1]:= \sum_{|\gamma|\le k}\big(
\|\langle \cdot \rangle^\rho \widetilde{\Gamma}^\gamma v_0\|_{L^\infty_{_{\Omega}}}
+\|\langle \cdot \rangle^\rho\widetilde{\Gamma}^\gamma \nabla_x v_0\|_{L^\infty_{_{\Omega}}}
+\|\langle \cdot \rangle^\rho\widetilde{\Gamma}^\gamma v_1\|_{L^\infty_{_{\Omega}}} \big);\\
\mathcal B_{\rho, k}[v_0,v_1]:= \sum_{|\gamma|\le k}\big(
\|\langle \cdot \rangle^\rho \widetilde{\Gamma}^\gamma v_0\|_{L^\infty}
+\|\langle \cdot \rangle^\rho\widetilde{\Gamma}^\gamma \nabla_x v_0\|_{L^\infty}
+\|\langle \cdot \rangle^\rho\widetilde{\Gamma}^\gamma v_1\|_{L^\infty} \big).
\end{eqnarray*}
These quantities will be used to control the influence of the initial data to the $L^\infty$ norms of the solution.
% Finally, in the sequel $\partial u$ without any index denotes one of the derivatives $\partial_t u$ or $\partial_{x_1} u$ or $\partial_{x_2}u$.

Using the vector fields in $\widetilde{\Gamma}$, we obtain the following
Sobolev-type inequality.
\begin{lemma}\label{KlainermanSobolev}\
Let $v \in C_0^2(\overline{\Omega})$.
Then we have
\begin{eqnarray*}
\sup_{x \in \Omega} |x|^{1/2}|v(x)|
\lesssim
\sum_{\substack{|\alpha|+\beta\le 2 \\ \beta \ne 2}}
\|\pa_x^\alpha \Lambda^\beta v\|_{L^2(\Omega)}.
\end{eqnarray*}
\end{lemma}

\noindent{\it Proof.}\ It is well known that for $w \in C_0^2(\R^2)$ we have
\begin{equation}
|x|^{1/2}|w(x)| \lesssim
 \sum_{\substack{|\alpha|+\beta \le 2\\ \beta\not=2}}\|\pa_x^\alpha \Lambda^\beta w\|_{L^2(\R^2)}, \quad  x\in \R^2
\label{KlainermanIneq}
\end{equation}
(see Klainerman \cite{kl0} for the proof).

Let $\chi=\chi(x)$ be a nonnegative smooth function satisfying
$\chi(x)\equiv 0$ for $|x|\le 1$ and $\chi(x)\equiv 1$ for $|x|\ge 2$. If we rewrite $v$ as
$v=\chi v+(1-\chi) v$, then
we have $\chi v\in C^\infty_0(\R^2)$ and \eqref{KlainermanIneq} leads to
$$
\sup_{x \in \Omega} |x|^{1/2}|v(x)|\lesssim
\sum_{\substack{|\alpha|+\beta \le 2\\ \beta\ne 2}} \|\pa_x^\alpha \Lambda^\beta (\chi v)\|_{L^2(\R^2)}+\|(1-\chi)v\|_{L^\infty(\Omega)}.
$$
By using the Sobolev embedding to estimate the last term, we arrive at
$$
\sup_{x \in \Omega} |x|^{1/2}|v(x)|\lesssim
\sum_{\substack{|\alpha|+\beta \le 2\\ \beta\not=2}} \|\pa_x^\alpha \Lambda^\beta v\|_{L^2(\Omega)}
+\sum_{|\alpha| \le 2} \|\pa_x^\alpha v\|_{L^2(\Omega)}.
$$
This completes the proof.
\hfill$\qed$
%%%%%%%%%%%%%%%%%%%%%%%%%%%%%%%%%%%%%%%%%%%%%%%%%%%%%%%%
\subsection{Elliptic estimates}
The following elliptic estimates will be used in the energy estimates.

\begin{lemma}\label{elliptic2}\
Let $R>1$, $m$ be an integer with $m \ge 2$ and $v \in H^m(\Omega)$ such that $\partial_\nu v=0$ on $\partial \Omega$.
Then we have
\begin{eqnarray}\label{eq.ap1}
 \|\pa^\alpha_x v\|_{L^2(\Omega)} \lesssim \|\Delta v\|_{H^{|\alpha|-2}(\Omega)}+\|v\|_{H^{|\alpha|-1}(\Omega_{R+1})}
\end{eqnarray}
for $2\le |\alpha| \le m$.
\end{lemma}

\begin{proof}
Let $\chi$ be a $C^\infty_0({\R}^n)$ function such that
$\chi(x) \equiv 1$ for $|x|\le R$ and $\chi (x)\equiv 0$ for $|x|\ge R+1$.
We set $v_1=\chi v$ and $v_2=(1-\chi) v$, so that $v=v_1+v_2$.

If we put $h=\Delta v_1$, the function $v_1$ solves the elliptic problem
\begin{equation*}
\left\{
\begin{array}{ll}
\Delta v_1= h & \text{ on } \Omega_{R+1}, \\
\partial_{\nu} v_1=0 & \text{ on } \partial \Omega, \\
v_1=0 & \text{ on } \partial B_{R+1}.
\end{array}
\right.
\end{equation*}
From Theorem 15.2 of \cite{ADN}, we have
\begin{equation}\label{eq.ADN}
\|v_1\|_{H^{l}(\Omega_{R+1})}\lesssim
\|h\|_{H^{l-2}(\Omega_{R+1})}+\|v_1\|_{L^{2}(\Omega_{R+1})}=
\|\Delta v_1
\|_{H^{l-2}(\Omega_{R+1})}+\|v_1\|_{L^{2}(\Omega_{R+1})}
\end{equation}
for $l\ge 2$.
Hence
\begin{eqnarray*}
\|\pa^\alpha_x v_1\|_{L^2(\Omega)}
& \lesssim & \|\Delta v\|_{H^{|\alpha|-2}(\Omega_{R+1})}
 +\|\nabla v\|_{H^{|\alpha|-2}(\Omega_{R+1})}+\|v\|_{H^{|\alpha|-2}(\Omega_{R+1})}\\
& \lesssim & \|\Delta v\|_{H^{|\alpha|-2}(\Omega_{R+1})}
 +\|v\|_{H^{|\alpha|-1}(\Omega_{R+1})}
\end{eqnarray*}

Now we consider $v_2$. Note that $v_2$ can be regarded as a function in $\R^2$ and we can write $\|\pa^\alpha_x v_2\|_{L^2(\Omega)}=\| \pa^\alpha_x v_2\|_{L^2(\R^2)}$.
Let us recall that
$\|\pa^\beta_x w\|_{L^2({\R}^n)}\lesssim \|\Delta w\|_{L^2({\R}^n)}$
for any $w \in H^2({\R}^n)$ and $|\beta|=2$.
Writing  $\alpha=\beta+\gamma$ with $|\beta|=2$ and $|\gamma|=|\alpha|-2$, we have
\begin{eqnarray*}\nonumber
 \|\pa^\alpha_x v_2\|_{L^2(\Omega)}&\lesssim &  \|\Delta \pa^\gamma_x v_2\|_{L^2(\R^2)}
\lesssim \|\Delta v_2\|_{ H^{|\alpha|-2}(\R^2)}\\
&\lesssim &\|\Delta v\|_{ H^{|\alpha|-2}(\Omega)}+\|v\|_{ H^{|\alpha|-1}(\Omega_{R+1})}.
\end{eqnarray*}
Combining this inequality with the estimate for $v_1$, we find \eqref{eq.ap1}.
\end{proof}

%%%%%%%%%%%%%%%%%%%%%%%%%%%%%%%%%%%%%%%%%%%%%%%%%%%%%%%%
\subsection{Decay estimates for the linear wave equation with Neumann boundary condition}\label{LinearNeumann}\rm
Given $T>0$, we consider the mixed problem
\begin{equation}\label{eq.PMfT}
\begin{array}{ll}
(\partial_t^2-\Delta) u =f, & (t,x) \in(0,T)\times \Omega,\\
\partial_\nu u(t,x)=0, & (t,x) \in(0,T)\times \partial\Omega,\\
u(0,x)=u_0(x), & x\in \Omega,\\
(\partial_t u)(0,x)=u_1(x), & x\in \Omega.
\end{array}
\end{equation}
It is known that for $u_0\in H^2(\Omega)$, $u_1\in H^1(\Omega)$ and
$f\in {\mathcal C}^1\bigl([0,T); L^2(\Omega)\bigr)$, the mixed problem \eqref{eq.PMfT}
admits a unique solution
$$
u\in \bigcap_{j=0}^2 {\mathcal C}^j\bigl([0,T); H^{2-j}(\Omega)\bigr),
$$
provided that $(u_0, u_1, f)$ satisfies the compatibility condition of order $0$, that is to say,
\begin{equation}
 \label{CCo0}
 \pa_\nu u_0(x)=0,\quad x\in \pa \Omega
\end{equation}
(see \cite{I68} for instance).
Under these assumptions for $\vec{u}_0:=(u_0,u_1)$,
the solution $u$ of \eqref{eq.PMfT} will be denoted by $S[\vec{u}_0,f](t,x)$.
We set $K[\vec{u}_0](t,x)$ for the solution of \eqref{eq.PMfT} with $f\equiv 0$ and
$L[f](t,x)$ for the solution of \eqref{eq.PMfT} with $\vec{u}_0\equiv (0,0)$; in other words we put
$$
K[\vec{u}_0](t,x):=S[\vec{u}_0, 0](t,x), \quad L[f](t,x):= S[(0,0),f](t,x)
$$
so that we get
$$
S[\vec{u}_0,f](t,x)=K[\vec{u}_0](t,x)+L[f](t,x),
$$
where $K[\vec{u_0}]$ and $L[f]$ are well defined because both of $(u_0, u_1, 0)$ and $(0,0,f)$ satisfy the compatibility condition of order $0$.
%\footnote{I remove Remak 2.1}
%\begin{rem}
%\normalfont
%It is known that the set
%$$
%D_N:=\{\vec{u}_0=(u_0, u_1)\in H^2(\Omega)\times H^1(\Omega);\, \pa_\nu u_0(x)=0,\ x\in \pa \Omega\}
%$$
%is dense in $H^1(\Omega)\times L^2(\Omega)$ (see \cite{I68}).
%
%\noindent
%Hence, using the energy identity,
%we can define $K[\vec{u}_0]\in \bigcap_{j=0}^1{\mathcal C}^j\bigl([0,T]; H^{1-j}(\Omega)\bigr)$
%for $\vec{u}_0\in H^1(\Omega)\times L^2(\Omega)$ by $K[\vec{u}_0]:=\lim_{n\to\infty}K[\vec{u}_{0,n}]$,
%where $\{\vec{u}_{0,n}\}$ is a sequence in $D_N$ approximating $\vec{u}_0$.
%\end{rem}
In order to obtain a smooth solution to \eqref{eq.PMfT}, we need the compatibility condition of infinite order.
\begin{definition}\label{def.complinear}
Suppose that $u_0$, $u_1$ and $f$ are smooth. Define $u_j$ for $j\ge 2$ inductively by
$$
u_j(x)=\Delta u_{j-2}(x)+(\partial_t^{j-2}f)(0,x),\quad j\ge 2.
$$
We say that $(u_0, u_1, f)$ satisfies the compatibility
condition of infinite order in
$\Omega$ for \eqref{eq.PMfT}, if one has
\begin{equation*}
\partial_\nu u_j=0 \quad\text{on}\quad \partial\Omega
\end{equation*}
for any nonnegative integer $j$.
\end{definition}
We say that $(u_0, u_1, f)\in X(T)$ if the following three conditions are
satisfied:
\begin{itemize}
\item $(u_0,u_1)\in \mathcal C_0^{\infty}(\overline{\Omega})
       \times \mathcal C_0^{\infty}(\overline{\Omega})$,
\item  $f\in C^{\infty}([0,T) \times \overline{\Omega})$;
      moreover,
 $f(t,\cdot)\in  \mathcal C_0^{\infty}(\overline{\Omega})$ for any $t\in [0, T)$,
\item $(u_0, u_1, f)$ satisfies the compatibility condition of infinite order.
\end{itemize}
It is known that if $(u_0, u_1, f)\in X(T)$, then we have $S[\vec{u}_0,f]\in {\mathcal C}^\infty\bigl([0,T)\times \overline{\Omega}\bigr)$ (see \cite{I68} for instance).

The following decay estimates play important roles in our proof of the main theorem.
%%%%%%%%%%%%%%%%%%%%%%%%%%%%%%%%%%%%%%%%%%%%%
\begin{theorem}\ \label{main}
Let ${\mathcal O}$ be a convex set and $k$ be a nonnegative integer.
Suppose that $\Xi=(\vec{u}_0,f)=({u}_0, {u}_1,f) \in X(T)$.

\noindent
{\rm (i)}\ Let $\mu>0$.
Then  we have
\begin{eqnarray}\label{ba3}
%&& \log^{-1}\left(\eu+\frac{\langle t+|x|\rangle }{\langle t-|x|\rangle }\right) |S[\Xi](t,x)|\lesssim \\
\sum_{|\delta|\le k} |\Gamma^\delta S[\Xi](t,x)|
\lesssim {\mathcal A}_{2+\mu,3+k} [\vec{u}_0]
+ \log(e+t)\sum_{|\delta|\le 3+k}\| |y|^{1/2}W_{1,1+\mu}(s,y)\Gamma^\delta  f(s,y)\|_{L^\infty_{t}L^\infty_\Omega} %onumber
\end{eqnarray}
for $(t,x)\in [0,T) \times {\overline{\Omega}}$.

\vspace{1mm}

\noindent
{\rm (ii)}\ Let $0<\eta<1/2$ and $\mu>0$. Then we have
\begin{eqnarray}
&&   w_{(1/2)-\eta}^{-1}(t,x) \sum_{|\delta|\le k} |\Gamma^\delta \partial S[\Xi](t,x)|\lesssim \nonumber\\
&& \quad \lesssim \mathcal A_{2+\mu,k+4}[\vec{u}_0]+
 \log^2(e+t+|x|)
 \sum_{|\delta|\le k+4}\||y|^{1/2}W_{1,1}(s,y)\Gamma^\delta f(s,y)\|_{L^\infty_{t}L^\infty_\Omega}, \label{ba4weak}
\\
&&   w_{1/2}^{-1}(t,x) \sum_{|\delta|\le k} |\Gamma^\delta \partial S[\Xi](t,x)|\lesssim
\nonumber\\
&& \quad \lesssim \mathcal A_{2+\mu,k+4}[\vec{u}_0]+
  \log^2(e+t+|x|)
 \sum_{|\delta|\le k+4}\||y|^{1/2}W_{1,1+\mu}(s,y)\Gamma^\delta f(s,y)\|_{L^\infty_{t}L^\infty_\Omega}
 \label{ba4}
\end{eqnarray}
for $(t,x)\in [0,T)\times {\overline{\Omega}}$.

\vspace{1mm}

\noindent
{\rm (iii)}\ Let $0<\eta<1$ and $\mu>0$. Then we have
\begin{eqnarray}
&& w_{1-\eta}^{-1}(t,x)
   \sum_{|\delta|\le k} |\Gamma^\delta \pa\partial_t S[\Xi](t,x)|\lesssim
\nonumber\\
&& \quad \lesssim  {\mathcal A}_{2+\mu,k+ 5}[\vec{u}_0]+
\log^{2} (e+t+|x|)\sum_{|\delta|\le k+ 5}\||y|^{1/2}W_{1,1}(s,y) \Gamma^\delta
f(s,y)\|_{L^\infty_{t}L^\infty_\Omega} \label{ba4t}
\end{eqnarray}
for $(t,x)\in [0,T)\times {\overline{\Omega}}$.
\end{theorem}

We will prove Theorem~\ref{main} in Section~\ref{sec.pointwise} below,
by using the so-called cut-off method to combine the corresponding decay estimates for
the Cauchy problem with the local energy decay.
%%%%%%%%%%%%%%%%%%%%%%%%%%%%%%%%%%%%%%%%%%%%%%%%%%%%%%%%

\section{The abstract argument for the proof of the
main theorem}\label{AAPMT}
Since the local existence of smooth solutions for the mixed problem
\eqref{eq.PMN} has been shown by \cite{SN89} (see also the Appendix),
what we need to do for showing the large time existence of the solution
is to derive suitable {\it a priori} estimates: following
\cite{SN89}, we need the control of $\|u(t)\|_{H^9(\Omega)}+\|\pa_t u(t)\|_{H^8(\Omega)}$ for the solution $u$.

Let $u$ be the local solution of \eqref{eq.PMN},
assuming \eqref{eq.smalldata.semi} holds for large $m\in \N$ and $s>0$.
Let $T^*$ be the supremum of $T$ such that \eqref{eq.PMN} admits a (unique) classical solution in $[0,T)\times \overline{\Omega}$.
For $0<T\le T^*$, a small $\eta>0$,
and nonnegative integers $H$ and $K$
we define
\begin{align*}
 \en_{H,K}(T)\equiv &
\sum_{|\gamma|\le H-1} \|w_{1/2}^{-1} \Gamma^\gamma \pa u\|_{L^\infty_TL^\infty_\Omega}
+\sum_{1\le j+ |\alpha|\le K}\|\partial_t^j \pa_x^\alpha u\|_{L^\infty_T L^2_{\Omega}}
\\
&
+\sum_{|\delta|\leq K-2} \| \jb{s}^{-1/2} \Gamma^\delta \partial u(s,y)\|_{L^\infty_T L^2_{\Omega}}
+\sum_{|\delta|\leq K-8} \| \jb{s}^{-(1/4)-\eta} \Gamma^\delta \partial u(s,y)\|_{L^\infty_T L^2_{\Omega}}
\\
&
+\sum_{|\delta|\leq K-14} \| \jb{s}^{-2\eta} \Gamma^\delta \partial u(s,y)\|_{L^\infty_T L^2_{\Omega}}
+\sum_{|\delta|\leq K-20} \| \Gamma^\delta \partial u\|_{L^\infty_T L^2_{\Omega}}.
\end{align*}
We neglect the first sum when $H=0$.
Similarly we
neglect summations taken over the empty set as $K$ varies.
We also put $$\en_{H,K}(0)=\lim_{T\to 0^+}\en_{H,K}(T).$$ Observe that $\en_{H,K}(0)$ can be determined only by $\phi$, $\psi$ and $G$ and that we have
$$
\en_{H,K}(0)\lesssim \|\phi\|_{H^{m+1,s}(\Omega)} +\|\psi\|_{H^{m,s}(\Omega)}
$$
for suitably large $m \in\N$ and $s>0$ depending on $H$ and $K$.
From \eqref{eq.smalldata.semi} for such $m\in \N$ and $s>0$, we see that $\en_{H,K}(0)$ is finite.
The previous inequality can be obtained combining the embedding
$H^r(\Omega)\hookrightarrow L^\infty(\Omega)$ for $r>1$
with the trivial inequality
$|\Gamma_3 f|\le \langle x\rangle|\partial_1 f|+\langle x\rangle|\partial_2 f|$ and the equivalence between
$\sum_{ |\alpha|\le m}\|\langle \cdot
\rangle^s  \pa_x^\alpha f\|_{ L^2_\Omega}$
and $\|\langle \cdot
\rangle^s  f\|_{H^m(\Omega)}$.
In order to optimize $m$ or $s$ it is possible to use sharpest embedding theorem in weighted Sobolev spaces proved for example in \cite{GeLu04}.

Our goal is to show the following claim.
\begin{claim}\label{claim}
\normalfont
We can take suitable $H$ and $K$ and sufficiently large $m$ and $s$, so that there exist
positive numbers
$C_1$, $P$ and $Q$ and
a strictly increasing continuous function ${\mathcal R}:[0,\infty)\to[0,\infty)$ with ${\mathcal R}(0)=0$ such that
if $\en_{H,K}(T) \le 1$, then
\begin{equation}\label{eq.energy}
  \en_{H,K}(T)\le C_1\varepsilon + %C_0
{\mathcal R}\left(\en_{H,K}^{P}(T)\log^{Q}(e+T)\right) (\varepsilon+\en_{H,K}(T)),
\end{equation}
provided that \eqref{eq.smalldata.semi} holds with $\ve\le 1$.
Here $C_1$, $P$, $Q$ and ${\mathcal R}$ are independent of $\ve$ and $T$.
\end{claim}

Let us explain how from \eqref{eq.energy}
we can gain the lifespan estimate.
Suppose that the above claim is true.
If we assume \eqref{eq.smalldata.semi}
for some $m$ and $s$ which are sufficiently large, then, as we have mentioned, there
exists $C_*>0$ such that
$\en_{H,K}(0)< 2C_*\ve$. We may assume $C_*\ge \max\{C_1,1\}$.
 We set $\ve_0=\min \{(2C_*)^{-1}, 1\}$ and suppose that $0<\ve\le \ve_0$, so that we have
$\ve\le 1$ and $2C_*\ve\le 1$.
We put
$$
 T_*(\ve):=\sup \left\{T\in [0, T^*):\, \en_{H,K}(T)\le 2C_*\ve\right\}.
$$
In particular, for any $T\le  T_*(\ve)$, we have $ \en_{H,K}(T)\le 1$.
From \eqref{eq.energy} with $T=T_*(\ve)$, we get
\begin{equation*}
\en_{H,K}\left(T_*(\ve)\right)\le C_*\varepsilon + {\mathcal R}\left((2C_* \varepsilon)^{P} \log^{Q}\left(e+T_*(\ve)\right)\right)
(3C_* \varepsilon).
\end{equation*}
We are going to prove
\begin{equation}
\label{eq.lifespan00}
{\mathcal R}\left((2C_* \varepsilon)^{P} \log^{Q}\left(e+T^*\right) \right) >\frac{1}{4}
\end{equation}
by contradiction. Suppose
that $T^*$ satisfies
\begin{equation} \label{eq.lifespan}
{\mathcal R}\left((2C_* \varepsilon)^{P} \log^{Q}\left(e+T^*\right) \right) \le \frac{1}{4}.
\end{equation}
Since $T_*(\ve)\le T^*$, and $\mathcal R$ is an increasing function,
we obtain
\begin{equation*}
\en_{H,K}\left(T_*(\ve)\right)
  \le \frac{7}{4} C_* \varepsilon
  <2C_* \varepsilon.
\end{equation*}
Therefore we get $T_*(\ve)=T^*$, because otherwise
the continuity of $\en_{H,K}(T)$ implies that
there exists $\widetilde{T}>T_*(\ve)$ satisfying $\en_{H,K}(\widetilde{T})\le 2C_*\ve$, which contradicts the definition of $T_*(\ve)$.
However, if $T_*(\ve)=T^*$, and $H,K$ are sufficiently large,
we can prove
\begin{align}
 \|u\|_{L^\infty_{T^*}H^9(\Omega)}+\|\pa_tu\|_{L^\infty_{T^*}H^8(\Omega)}
& \lesssim
\ve + (1+T^*) \en_{H,K}(T^*) \label{eq.stima98}\\
&=\ve + (1+T_*(\ve)) \en_{H,K}\left(T_*(\ve)\right)
 \lesssim
\ve + (1+T_*(\ve)) 2C_*\ve, \nonumber
\end{align}
and we can extend the solution beyond the time $T^*$
by the local existence theorem, which contradicts the definition of
$T^*$. Therefore \eqref{eq.lifespan} is not true,
and we obtain \eqref{eq.lifespan00}.
This means that,
for any $\varepsilon \le \varepsilon_0$,
there exists $\tilde{C}>0$ such that
\begin{equation}\label{eq.lifePQ}
T^*>\exp\{\tilde C\epsilon^{-P/Q}\}.
\end{equation}
It remains to show \eqref{eq.stima98}.
It is evident that
$$
\|u\|_{L^\infty_{T^*}H^9(\Omega)}+\|\pa_tu\|_{L^\infty_{T^*}H^8(\Omega)}\lesssim \|u\|_{L^\infty_{T^*}L^2_\Omega}+\en_{0,9}(T^*).
$$
In order to estimate $\|u\|_{L^\infty_{T^*}L^2_\Omega}$ we will use the expression
$$
u(t,x)=u(0,x)
 +\int_0^t \partial_t u(\tau, x) d\, \tau,
$$
which leads to
$$
\|u\|_{L^\infty_{T^*}L^2_\Omega}\lesssim
\ve +T^*\en_{0,1}(T^*).
$$

As a conclusion, we obtain
\eqref{eq.lifespanT1.semi}, once we can show
that Claim~\ref{claim} is true
with $P=Q=1$.
This will be done in the next three sections.

\section{Energy estimates for the standard derivatives}
In this section we are going to estimate
$\|\pa_t^j\pa_x^\alpha u\|_{L^\infty_TL^2_\Omega}$ for $j+|\alpha|\ge 1$.
In the first subsection, we consider the case where $j\ge 0$ and $ |\alpha|=1$. This can be done directly through the standard energy inequalities.
In the second subsection, the case where $j\ge 1$ and $|\alpha|\ge 2$
will be treated with the help of the elliptic estimate, Lemma~\ref{elliptic2}.
In the third subsection, we consider the case where $j=0$ and $|\alpha|\ge 2$. Lemma~\ref{elliptic2} will be used again, but this time we need the estimate of $\| u\|_{L^\infty_TL^2(\Omega_{R+1})}$ for some $R>0$,
which is not included in the definition of $\en_{H,K}(T)$. Since we
are considering the $2$D Neumann problem, it seems difficult
to use some embedding theorem to estimate $\| u\|_{L^\infty_TL^2(\Omega_{R+1})}$ by $\|\nabla_x u \|_{L^\infty_TH^k(\Omega)}$ with some positive integer $k$.
Instead, we will employ the $L^\infty$ estimate, Theorem~\ref{main},
for this purpose.
\subsection{On the energy estimates for the derivatives in time}
First we set
\begin{equation*}
E(v;t)=\frac12 \int_{\Omega} \{|\partial_t v(t,x)|^2+|\nabla_x v(t,x)|^2\} dx
\end{equation*}
for a smooth function $v=v(t,x)$.

Let $j$ be a nonnegative integer.
Since $\partial_t$ commutes with the restriction of the function to $\partial \Omega$, we have
$\partial_\nu \partial_t^j u(t,x)=0$ for all $(t,x)\in (0,T) \times \partial \Omega$.
Therefore, by the standard energy method, we find
\begin{equation*}
\frac{d}{dt} E(\partial_t^j u;t)=
\int_{\Omega} \partial_t^j( G(\partial u))(t,x)\,\partial_t^{j+1} u(t,x) dx.
\end{equation*}

Recalling the definition of $\en_{H,K}(T)$, for $j+|\alpha|\ge 1$ we have
\begin{equation}\label{eq.nu}
|\partial^j_t\nabla_x^\alpha u(t,x)|\le w_{1/2}(t,x)
\en_{j+|\alpha|, 0}(T),
 \quad x\in \Omega, t\in [0,T).
\end{equation}
Applying \eqref{eq.nu} and the Leibniz rule we find
$$
\frac{d}{dt} E(\partial_t^j u;t)\lesssim
\|w_{1/2}(t)\|_{L^\infty_\Omega}^2\en_{[j/2]+1,0}^2(T)\sum_{h=0}^{j}
\int_{\Omega} |\partial_t^h \partial u(t,x) |\,|\partial_t^{j+1} u(t,x) |dx.
$$
It is also clear that if $j+|\alpha| \ge 1$,  one has
$$
\|\partial_t^j\partial_x^\alpha u(t)\|_{L^2_{\Omega}} \le
 \en_{0,j+|\alpha|} (T), \quad t\in [0,T).
$$
This gives
$$
\frac{d}{dt} E(\partial_t^j u;t)\lesssim
\|w_{1/2}(t)\|_{L^\infty_\Omega}^2 \en^2_{[j/2]+1,0}(T) \en^2_{0,j+1}(T).
$$
Since $\en_{H,K}(T)$ is increasing in $H$ and $K$, we get
$$
\frac{d}{dt} E(\partial_t^j u;t)\lesssim \|w_{1/2}(t)\|_{L^\infty_\Omega}^2 \en^4_{[j/2]+1,j+1}(T).
$$
As a trivial consequence of \eqref{eq.ome},
we find $w_{1/2}(t,x)\le \langle t\rangle^{-1/2}$, so that
\begin{equation*}
\frac{d}{dt} E(\partial_t^j u;t)\lesssim \langle t \rangle^{-1}\en^4_{[j/2]+1,j+1}(T).
\end{equation*}
After integration this gives
\begin{eqnarray}\label{eq.ap9}
\sum_{l=0}^{j} \|\partial_t^{l+1} u(t)\|_{L^2_\Omega}+
\sum_{l=0}^{j} \|\partial_t^{l} \nabla_x u(t)\|_{L^2_\Omega}
\lesssim \en_{j+1}(0)+\en^2_{j+1}(T)\log^{1/2}(e+t)
\end{eqnarray}
for any $j\ge 0$ and $t \in [0,T)$, where
$$\en_s(T)=\en_{[(s-1)/2]+1,s}(T)$$ for any integer $s\ge 0$.

\subsection{On the energy estimates for the space-time derivatives.}\label{sec.32}
Since the spatial derivatives do not preserve
the Neumann boundary condition,  we need to use
elliptic regularity results.

We shall show that for $j\ge 1$ and $k\ge 0$ it holds
\begin{eqnarray}\label{eq.ap11}
\sum_{|\alpha|=k} \|\partial_t^{j} \partial_x^\alpha u(t)\|_{L^2_\Omega}
 \lesssim \en_{j+k}(0)+\en^2_{j+k}(T)\log^{1/2}(e+T)+\en^3_{j+k-1}(T)
\end{eqnarray}
with $\en_s(T)=\en_{[(s-1)/2]+1, s}(T)$
as before.

It is clear that (\ref{eq.ap11}) follows from (\ref{eq.ap9}) when $j\ge 1$ and $k=0,1$.

Next we suppose that
(\ref{eq.ap11}) holds for $j\ge 1$ and $k\le l$
with some positive integer $l$.
Let $|\alpha|=l+1$ and $j \ge 1$.
Since $|\alpha|\ge 2$,
we apply to $\partial_t^{j} u$ the elliptic estimate
(Lemma~\ref{elliptic2}) and we obtain
$$
    \|\partial_x^\alpha \partial_t^j u(t)\|
\lesssim \|\Delta \partial_t^{j} u(t)\|_{H^{l-1}(\Omega)}
    +\|\partial_t^{j} u(t)\|_{H^{l}(\Omega)}.
$$
By \eqref{eq.ap11} for $k\le l$, we see
that the second term has the desired bound.
On the other hand, using the fact that $u$ is a solution to \eqref{eq.PMN}, for the first term we have
\begin{eqnarray}\nonumber
\|\Delta \partial_t^{j} u(t)\|_{H^{l-1}(\Omega)}\lesssim \|\partial_t^{j+2} u(t)\|_{H^{l-1}(\Omega)}
 + \|\partial_t^{j} (G(\partial u))(t)\|_{H^{l-1}(\Omega)}.
\end{eqnarray}
Since $(j+2)+(l-1)=j+l+1$, it follows from \eqref{eq.ap11} for $k=l-1$
with $j$ replaced by $j+2$ that
$$
\|\pa_t^{j+2} u(t)\|_{H^{l-1}(\Omega)}\lesssim
\en_{j+l+1}(0)+\en^2_{j+l+1}(T)\log^{1/2}(e+T)+\en^3_{j+l}(T),
$$
which is the desired bound.
Finally,  observing that
$w_{1/2}(t,x)\le 1$, we get
\begin{align*}
\|\partial^j_t G(\partial u)(t) \|_{H^{l-1}(\Omega)}\lesssim
\sum_{1\le |\beta|\le [(j+l-1)/2]+1}\|\partial^\beta u(t)\|_{L^\infty_\Omega}^2
 \sum_{1\le |\gamma|\le j+l}\| \partial^\gamma u(t)\|_{L^2_\Omega}
\lesssim \en^3_{j+l}(T).
\end{align*}
Combining these estimates, we obtain \eqref{eq.ap11} for $j\ge 1$ and $k=l+1$.
This completes the proof of \eqref{eq.ap11} for $j\ge 1$ and $k\ge 0$.
% It remains to estimate $\|\partial_t^{j+2} \partial^\beta u(t)\|$ with $0\le |\beta|\le |\alpha|-2$. If $|\beta|\le |\alpha|-3$, we apply the inductive assumption.
% For the case $|\beta|=|\alpha|-2$ we repeat the same argument arriving at
% \begin{eqnarray*}
% \|\partial_t^{j+2} \partial^\beta  u(t)\| &\lesssim & \|\Delta \partial_t^{j+2} u\|_{H^{|\alpha|-4}(\Omega)}+\| \partial_t^{j+2} u \|_{H^{|\alpha|-3}(\Omega)}\\
% &\lesssim & \|\partial_t^{j+4} u(t)\|_{H^{|\alpha|-4}(\Omega)}
%  + \|\partial_t^{j+2} (G(\partial u))(t)\|_{H^{|\alpha|-4}(\Omega)}+\| \partial_t^{j+2} u \|_{H^{|\alpha|-3}(\Omega)}.
% \end{eqnarray*}
% In a finite number $N$ of these steps we reach  $\|\partial_t^{j+N} u(t)\|$ or  $\|\partial_t^{j+N} \nabla u(t)\|$ and conclude the estimate
% by means of \eqref{eq.ap9}.

\subsection{On the energy estimates for the space derivatives}\label{sec.spaceder}
Our aim here is to %show \eqref{eq.ap11} for $j=0$ and $|\alpha| \ge 1$.
estimate $\|\partial_x^\alpha u\|_{L^\infty_TL^2_{\Omega}}$ for $|\alpha|=k\ge 1$.
The estimate for $k=1$ is included in \eqref{eq.ap9}.
Let us consider the case $|\alpha|=k\ge 2$.
Let us fix $R>1$.
The elliptic estimate \eqref{eq.ap1} gives
\begin{eqnarray*}
\sum_{|\alpha|=k}
\|\partial_x^\alpha u\|_{L^\infty_T L^2_{\Omega}}
 &\lesssim & \|\Delta u\|_{L^\infty_T H^{k-2}(\Omega)}+\|u\|_{L^\infty_TH^{k-1}(\Omega_{R+1})}\\
&\lesssim & \|\partial_{t}^2 u\|_{L^\infty_T H^{k-2}(\Omega)}+\|G(\partial u)\|_{L^\infty_T H^{k-2}(\Omega)}+
\|u\|_{L^\infty_T H^{k-1}(\Omega_{R+1})}.
\end{eqnarray*}
The first term can be estimated by \eqref{eq.ap11} and we get
$$
 \|\partial_t^2 u\|_{L^\infty_T H^{k-2}(\Omega)}
 \lesssim \en_{k}(0)+\en^2_{k}(T)\log^{1/2}(e+T)+\en^3_{k-1}(T).
$$
For the second term, we obtain the following inequality
as before:
$$
\|G(\partial u)\|_{L^\infty_T H^{k-2}(\Omega)}
\lesssim \en_{k-1}^3(T).
$$
As for the third term, we get
\begin{align}
\|u\|_{L^\infty_TH^{k-1}(\Omega_{R+1})}\lesssim &
\sum_{1\le |\beta|\le k-1}\|\pa_x^\beta u\|_{L^\infty_TL^2(\Omega_{R+1})}
+\|u\|_{L^\infty_TL^2(\Omega_{R+1})}
\nonumber\\
\lesssim &
\sum_{1\le |\beta|\le k-1}\|\pa_x^\beta u\|_{L^\infty_TL^2_\Omega}
+\|u\|_{L^\infty_TL^\infty(\Omega_{R+1})}.
\nonumber
\end{align}
Now we fix $\mu\in (0, 1/2)$ and use \eqref{ba3} with $k=0$ to obtain
\begin{align}
\|u\|_{L^\infty_TL^\infty_\Omega}
\lesssim & {\mathcal A}_{2+\mu, 3}[\phi, \psi]
+\log(e+T) \sum_{|\delta|\le 3}\left\|\jb{y}^{1/2}W_{1,1+\mu}(s,y)
\Gamma^\delta \left(G(\pa u)\right)(s,y)\right\|_{L^\infty_TL^\infty_\Omega}.
\label{eq.infty11a}
\end{align}
By using \eqref{eq.ome}, for any $ s\in [0,T)$ we have
$$
\sum_{|\delta|\le 3}
|\Gamma^{\delta}G(\partial u)( s,y)|
\lesssim
\langle s+|y|\rangle^{-3/2}\left(\min\{ \langle y \rangle, \langle |y|- s \rangle\}\right)^{-3/2} \en^3_{4,0}(T).
$$
This implies
$$
\sum_{|\delta|\le 3}\left\| |y|^{1/2} W_{1,1+\mu}(s,y) \Gamma^\delta \left(G(\partial u)\right)(s,y) \right\|_{L^\infty_T L^\infty_{\Omega}} \lesssim
 \en^3_{4,0}(T),
$$
and \eqref{eq.infty11a} gives
\begin{equation}
\|u\|_{L^\infty_TL^\infty_\Omega}
%\|u\|_{L^\infty_TL^\infty(\Omega_{R+1})}
\lesssim \mathcal A_{2+\mu, 3}[\phi, \psi]+\en^3_{4,0}(T)\log (e+T).
\label{eq.inftyOK}
\end{equation}
%In order to conclude the  energy estimates for the space derivatives it suffices to observe that
%$\mathcal A_{0, 3}[u_0,u_1]\lesssim e_{3,0,\nu}(0)$.\\
Summing up the estimates above, for $|\alpha|=k\ge 2$, we get
\begin{align*}
\sum_{|\alpha|=k} \|\pa_x^\alpha u\|_{L^\infty_TL^2_\Omega}
\le &{\mathcal A}_{2+\mu,3}[\phi, \psi]
{}+{\mathcal E}_{k}(0)+{\mathcal E}_{k}^2(T) \log^{1/2}(e+T)
{}+{\mathcal E}_{k-1}^3(T)+{\mathcal E}_{4,0}^3\log(e+T)\\
& {}+\sum_{1\le |\alpha|\le k-1}  \|\pa_x^\alpha u\|_{L^\infty_TL^2_\Omega}.
\end{align*}
Finally we inductively obtain
$$
\sum_{|\alpha|=k} \|\pa_x^\alpha u\|_{L^\infty_TL^2_\Omega}
\le {\mathcal A}_{2+\mu,3}[\phi, \psi]
{}+{\mathcal E}_{k}(0)+{\mathcal E}_{k}^2(T) \log^{1/2}(e+T)
{}+{\mathcal E}_{k-1}^3(T)+{\mathcal E}_{4,0}^3(T)\log(e+T)
$$
for $k\ge 1$.

\subsection{Conclusion for the energy estimates of the standard derivatives}
If $m$ and $s$ are sufficiently large, \eqref{eq.smalldata.semi}
and the Sobolev embedding theorem lead to
$$
{\mathcal A}_{2+\mu,3}[\phi,\psi]+{\mathcal E}_{K}(0) \lesssim \|\phi\|_{H^{m+1,s}(\Omega)}+\|\psi\|_{H^{m,s}(\Omega)} \lesssim \ve.
$$
Summing up the estimates in this section, we get
\begin{equation}
\sum_{1\le j+|\alpha|\le K} \|\pa_t^j \pa_x^\alpha u\|_{L^\infty_TL^2_\Omega}
\lesssim \ve+
{\mathcal E}_{K}^2(T) \log^{1/2}(e+T)
{}+{\mathcal E}_{K}^3(T)\log(e+T)
\label{Concl01}
\end{equation}
for each $K\ge 7$.

\section{On the energy estimates for the generalized derivatives}\label{sec.eegdI}
Throughout this section and the next one, we suppose that $K$ is sufficiently large,
and we assume that $\en_K(T)\le 1$.
\subsection{Direct energy estimates for the generalized derivatives}
Let $|\delta|\le K-2$.
Recalling (\ref{eq.commute}), it follows that
\begin{eqnarray}
\frac{d}{dt} E(\Gamma^\delta u;t) &=&
\int_{\Omega} \Gamma^\delta G(\partial u)(t,x)\,\partial_t \Gamma^\delta u(t,x) dx
\nonumber\\
&&+\int_{\partial \Omega} \nu\cdot \nabla_x \Gamma^\delta u(t,x)\,\partial_t \Gamma^\delta u(t,x) dS=:I_{\delta}(t)+I\!I_{\delta}(t),
\label{eq.eegd}
\end{eqnarray}
where $\nu=\nu(x)$ is the unit outer normal vector at $x \in \partial \Omega$
and $dS$ is the surface measure on $\partial \Omega$.

Since $G(\partial u)$ is a homogeneous polynomial of order three,  we can say that
\begin{equation}\label{eq.eegdn}
|\Gamma^\delta G(\partial u)\,\partial_t \Gamma^\delta u| \lesssim \sum_{|\delta_1|\le [|\delta|/2]}|\Gamma^{\delta_1} \pa u|^2
 \sum_{|\delta_2|\le |\delta|}|\Gamma^{\delta_2} \pa u(t,x)|^2.
\end{equation}
Applying the H\"older inequality and taking the $L^\infty$ norm of the first factor,
we arrive at
\begin{equation}\label{eq.eegdI}
|I_{\delta}(t)| \lesssim \langle t\rangle^{-1}{\mathcal \en}^2_{
 [|\delta|/2]+1, 0}(T) \jb{t}
 \en^2_{0, %|\delta|
 K }(T) \lesssim \en^4_{K}(T),
\end{equation}
since $|\delta|\le K-2$.

Now we treat the boundary term, by means of the trace theorem.
Since $\partial \Omega \subset B_{1}$, the norms of the generalized derivatives
on $\pa \Omega$
are equivalent to the norms of the standard derivatives. Hence for all $t\in (0,T)$ we have
$$
|I\!I_{\delta}(t)|\lesssim \sum_{1\le |\gamma|+k \le |\delta|+1} \|\partial_t^{k} %\nabla_x^\gamma
 \pa_x^\gamma u(t)\|^2_{L^2(\partial \Omega)}.
$$
Moreover, by the trace theorem and \eqref{Concl01}, we see that
$$
|I\!I_{\delta}(t)|
\lesssim \sum_{1\le |\gamma|+k \le |\delta|+2}
   \|\partial_t^{k} \pa_x^\gamma u(t)\|_{L^2_{\Omega}}^2
\lesssim
\left( \ve+{\mathcal R}_0( {\mathcal E}_{K}(T) \log^{1/2}(e+T)) \en_K(T) \right)^2,
$$
because of the assumption $|\delta|\le K-2$.
Here we put $${\mathcal R}_0(s)=s+s^2.$$
Summarizing the above estimates, for any $K\ge 7$ and
$|\delta|\le K-2$, it holds
\begin{align*}
\frac{d}{dt} E(\Gamma^\delta u;t)
&\lesssim
\left(
\ve+{\mathcal R}_0( {\mathcal E}_{K}(T) \log^{1/2}(e+T)) \en_K(T)
\right)^2
+\en_{K}^4(T)
\\
& \lesssim
\left(
\ve+{\mathcal R}_0( {\mathcal E}_{K}(T) \log^{1/2}(e+T)) \en_K(T) \right)^2.
\end{align*}
For the last inequality, we recall that
%$|\delta|+1\le K$ and
$\en_K(T)\le 1$.
After integration, %for $j+|\alpha|+\beta\le K-1$ with $j+|\alpha|\ge 1$,
this gives
\begin{align}
\sum_{|\delta| \le K-2} \| \Gamma^\delta \pa u(t)\|_{L^2_\Omega}
%\|\partial_t^j \pa^\alpha_x \Lambda^\beta u(t)\|_{L^2_{\Omega}}
 & \lesssim \en_{K}(0)+
t^{1/2} \left(\ve+{\mathcal R}_0( {\mathcal E}_{K}(T) \log^{1/2}(e+T)) \en_K(T)\right)
\notag \\
 & \lesssim \jb{t}^{1/2} \left(\ve+{\mathcal R}_0( {\mathcal E}_{K}(T) \log^{1/2}(e+T)) \en_K(T)\right).
\label{eq.seceegdI}
\end{align}

\subsection{Refinement of the energy estimates for the generalized derivatives}

Let $1\le |\delta|\le K-8$.
Since $\pa \Omega$ is a bounded set, it follows from \eqref{eq.eegd} that
\begin{align*}
|I\!I_{\delta}(t)|\lesssim &   \| \Gamma^\delta \partial_t u(t)\|_{L^2(\partial \Omega)}
\sum_{|\gamma|\le |\delta|} \| \Gamma^\gamma \nabla_x u(t)\|_{L^2(\partial \Omega)}\\
\lesssim & \sum_{1\le |\gamma|\le \delta} \|\pa^\gamma\pa_t u(t)\|_{L^\infty(\pa \Omega)}
\sum_{|\gamma|\le |\delta|} \| \pa^\gamma \nabla_x u(t)\|_{L^\infty(\partial \Omega)}.
\end{align*}
Since we have $|x|\le 1$ for $x\in \pa \Omega$,
we get $\langle |x|+t\rangle {} \simeq {}\langle t\rangle \simeq
\langle |x|-t\rangle$ for $x\in \pa\Omega$.
In particular we get $\sup_{x\in \pa\Omega} w_\nu(t,x) \lesssim \langle t\rangle^{-\nu}$
for $0<\nu \le 1$.
We fix sufficiently small and positive constants
$0<\eta<1/4$ and $\mu>0$.
Applying the pointwise estimates \eqref{ba4weak} and \eqref{ba4t} in Theorem~\ref{main}, we get
$$
|I\!I_{\delta}(t)|\lesssim \langle t\rangle^{-(3/2)+\eta}
\log^4 (e+t)
  \left(\mathcal A^2_{2+\mu,|\delta|+4}[\phi,\psi]+A^{2}_{|\delta|+4}(t)\right),
$$
where
$$
 A_{s}(t)=\sum_{|\gamma|\le s}\left\|\,|y|^{1/2}W_{1,1}(s,y)\Gamma^\gamma \left(G(\pa u)\right)(s,y) \right\|_{L^\infty_tL^\infty_\Omega}.
$$
If $m$ and $s$ are sufficiently large, by the Sobolev embedding theorem we have
$A_{2+\mu,|\delta|+4}[\phi,\psi]\lesssim \ve$ and we obtain
\begin{equation}\label{eq.IIA1}
|I\!I_{\delta}(t)|\lesssim \langle t\rangle^{-(3/2)+\eta} \log^4 (e+t)
  \left(\ve^2+A^2_{|\delta|+4}(t)\right).
\end{equation}
%being $A_1(t)$ the function in Proposition \ref{prop.pwe1}.
% with $f(t,x)=G(\partial u)$, $N=|\delta|$.
In order to estimate $A_{|\delta|+4}(t)$, we argue as in \eqref{eq.eegdn}, so that
$$
\sum_{|\gamma|\le |\delta|+4} |\Gamma^\gamma G(\partial u)(s,y)|
 \lesssim w_{1/2}^2(s,y) \en^2_{[(|\delta|+4)/2]+1,0}(T)
 \sum_{|\gamma'|\le |\delta|+4} |\Gamma^{\gamma'} \partial u(s,y)|.
$$
Now using \eqref{eq.ome} and applying Lemma~\ref{KlainermanSobolev} to
estimate $|\Gamma^{\gamma'} \partial u|$, we obtain
$$
 \sum_{|\gamma|\le|\delta|+4}
|\Gamma^\gamma G(\partial u)(s,y)|\lesssim
 |y|^{-1/2} W_{1,1}^{-1}(s,y)\en^2_{[(|\delta|+4)/2]+1,0}(T)
 \sum_{|\gamma|\le |\delta|+6} {\|\Gamma^{\gamma} \partial u(s,\cdot)\|_{L^2_\Omega}},
$$
which yields
\begin{equation}
 \label{Est_A_1}
A_{|\delta|+4}(t)
\lesssim \en^2_{K}(T)
 \sum_{|\gamma|\le |\delta|+6} \|\Gamma^{\gamma}
  \partial u(s,y)\|_{L^\infty_t{L^2_\Omega}}
%\lesssim \en^3_{K}(T) \jb{t}^{1/2},
\end{equation}
because we have $[(|\delta|+4)/2]\le [(K-1)/2]$ %and $|\delta|+6\le K-2$
for $|\delta|\le K-8$.
Observing that
$$
 \sum_{|\gamma|\le |\delta|+6} \|\Gamma^{\gamma}
  \partial u(s,y)\|_{L^\infty_t{L^2_\Omega}}
\lesssim \jb{t}^{1/2} \en_{K}(T)
$$
for $|\delta|\le K-8$, we see from \eqref{eq.IIA1} and \eqref{Est_A_1} that
$$
|I\!I_{\delta}(t)| \lesssim
 \langle t\rangle^{-(1/2)+2\eta} \left(\ve^2+\en^6_{K}(T)\right).
$$
Moreover for $|\delta|\le K-8$ the inequality \eqref{eq.eegdI} can be improved as
$$
|I_{\delta}(t)| \lesssim \langle t\rangle^{-1}{\mathcal \en}^2_{
 [|\delta|/2]+1, 0}(T) \left(\jb{t}^{1/4+\eta} \en_{0,K}
(T)\right)^2 \lesssim  \langle t\rangle^{-(1/2)+2\eta}  \en^4_{K}(T).
$$
Coming back to \eqref{eq.eegd}, one can conclude
from the assumption $\en_K(T)\le 1$ that
\begin{eqnarray}
\sum_{1 \le|\delta| \le K-8} \| \Gamma^\delta \pa u(t)\|_{L^2_\Omega}
%\|\partial_t^j\nabla^\alpha_x\Lambda^\beta u(t)\|_{L^2_\Omega}
&\lesssim& \en_{K}(0)
 +\jb{t}^{(1/4)+\eta} \left(\ve+\en^2_{K}(T)\right)
\nonumber
 \\
\label{eq.inftyeegd}
&\lesssim &
  \jb{t}^{(1/4)+\eta} \left(\ve+\en^2_{K}(T)\right).
\end{eqnarray}

%%%%%%%%%%%%%%%%%%%

Next step is to improve this estimate for lower $|\delta|$ in order to avoid the polynomial growth in $t$.
Let $1\le|\delta|\le K-14$. From \eqref{Est_A_1} and the definition of $\en_K(T)$
we get
$$
A_{|\delta|+4}(t)\lesssim \en^3_K(T) \jb{t}^{(1/4)+\eta}.
$$
From \eqref{eq.IIA1}, it follows that
\begin{eqnarray*}
|I\!I_{\delta}(t)| &\lesssim & \langle t \rangle^{-(3/2)+\eta} \log^{4} (e+t)
    \left(\varepsilon^2 + \langle t\rangle^{(1/2)+2\eta} \en^6_K(T)\right) \\
&\lesssim&
 \langle t \rangle^{-1+4\eta} \left(\varepsilon^2+\en^6_K(T)\right).
\end{eqnarray*}
On the other hand, for $|\delta|\le K-14$
it holds
$$
|I_{\delta}(t)| \lesssim \langle t\rangle^{-1}{\mathcal \en}^2_{
 [|\delta|/2]+1, 0}(T) \left(\jb{t}^{2\eta} \en_{0,K}(T)\right)^2 \lesssim  \langle t\rangle^{-1+4\eta}  \en^4_{K}(T).
$$
Summing up these estimates and integrating \eqref{eq.eegd},
we get
\begin{equation}\label{eq.seceegdIIj00}
\sum_{1\le |\delta| \le K-14} \| \Gamma^\delta \pa u(t)\|_{L^2_\Omega}
\lesssim \jb{t}^{2\eta} \left(\varepsilon +\en^2_K(T) \right).
\end{equation}

We repeat the above procedure once again
with $1\le|\delta|\le K-20$.
Being $|\delta|+6\le K-14$, from \eqref{Est_A_1}
we have $A_{|\delta|+4}(t)\lesssim \jb{t}^{2\eta}\en^3_K(T)$. In turn this implies
\begin{eqnarray*}
|I\!I_{\delta}(t)| & \lesssim &
\langle t \rangle^{-(3/2)+\eta} \log^{4} (e+t)
    \left(\varepsilon^2 + \langle t\rangle^{4\eta} \en^6_K(T)\right)\\
& \lesssim & \jb{t}^{-(3/2)+6\eta} \left(\ve^2+\en_K^6(T)\right).
\end{eqnarray*}
In this case $I_{\delta}(t)\le \langle t \rangle^{-1}\en_K^4(T)$. After integration
we get
\begin{eqnarray}
\sum_{ 1\le|\delta| \le K-20} \| \Gamma^\delta \pa u(t)\|_{L^2_\Omega}
%\|\partial_t^j\nabla^\alpha_x\Lambda^\beta u(t)\|_{L^2_{\Omega}}
&\lesssim & \en_{K}(0)+\en^2_{K}(T)\log^{1/2}(e+t)+\ve+\en^3_K(T)
\nonumber\\
&\lesssim &\varepsilon + \en^2_K(T)\log^{1/2}(e+t).
\label{eq.eegdfinal}
\end{eqnarray}
This estimate is the best we can obtain with our methods due to the estimate of $I_{\delta}(t)$.

\section{Boundedness for the $L^\infty$ norm
and the conclusion of the proof of Theorem~\ref{thm.mainsemi}}

Summarizing \eqref{Concl01}, \eqref{eq.seceegdI}, \eqref{eq.inftyeegd}, \eqref{eq.seceegdIIj00}, \eqref{eq.eegdfinal} we have
\begin{equation}\label{eq.e0K}
\en_{0,K}(T)\lesssim \ve +\mathcal R_0(\en_{[(K-1)/2]+1,K}(T) \log^{1/2}(e+T))\en_{[(K-1)/2]+1,K}(T)
\end{equation}
with $K\ge 20$ and $\mathcal R_0(s)=s+s^2$.
If $\en_{H,0}(T)$ with $H=[(K-1)/2]+1$ has the same bound of $\en_{0,K}(T)$
given in \eqref{eq.e0K}, then we conclude that the estimate \eqref{eq.energy}
in the Claim~\ref{claim} holds for $P=1$ and $Q=1/2$, and hence $T^* \ge
\exp(\tilde C \epsilon^{-2})$.
However, $\mathcal R_0$ (and hence $Q$) will be changed due to the
following argument. Such a modification yields a worse estimate for the lifespan.

Since we assume $\phi,\psi \in \mathcal{C}_0^\infty(\overline{\Omega})$,
there is a positive constant $M$ such that $|x|\le t+M$ in $\supp u(t,\cdot)$ for $t\ge 0$.
Hence we have $\log(e+t+|x|)\lesssim \log(e+t)$ in $\supp u(t, \cdot)$.

From \eqref{Est_A_1} and the definition of $\en_K(T)$,
it follows that
$A_{|\delta|+4}(t)\lesssim \en^3_K(T)$
for $K\ge 26$ and $|\delta|\le K-26$.
Let $\mu>0$.
Then we have ${\mathcal A}_{2+\mu, K-22}[\phi,\psi]\lesssim \ve$ if $m$ and $s$ are sufficiently large.
For fixed $0<\eta<1/2$, by \eqref{ba4weak},  we obtain
\[
\sum_{|\gamma|\le K-26}| \Gamma^\gamma \partial u(t, x)|
\lesssim \mathcal B(\varepsilon,t)  \,w_{(1/2)-\eta}(t,x)
\]
where
\[
\mathcal B(\varepsilon,t) :=\ve +\log^{2}(e+t)\en_K^3(T).
\]
Using this estimate, we obtain
\[
\sum_{|\gamma|\le K-26}|\Gamma^\gamma G(\partial u)(t,x)|\lesssim w_{1/2}^2(t,x)\en^2_{[(K-1)/2]+1,0}(T)w_{(1/2)-\eta}(t,x)\mathcal B(\varepsilon,t).
\]
Since $|y|^{1/2} w_{1/2-\eta}\lesssim 1$,  this implies
\[
\mathcal A_{|\delta|+4}(t)\lesssim \en_K^2(T)\mathcal B(\varepsilon,t)
\]
for any $|\delta|+4\le K-26$.  %and $K\ge 31$.
Therefore, \eqref{ba4} in Theorem~\ref{main} yields
\begin{eqnarray*}
\sum_{|\gamma|\le K-30}| \Gamma^\gamma \partial u(t, x)|\lesssim
  \left( \varepsilon+ \mathcal B(\varepsilon,t) \en^2_{K}(T)\log^2(e+t)
   \right) w_{1/2}(t,x).
\end{eqnarray*}
For $K\ge 61$ we have $[(K-1)/2]+1 \le K-30$, and we conclude that
\begin{equation}
\label{Concl02}
\sum_{|\gamma|\le [(K-1)/2]+1} \|w_{1/2}^{-1} \Gamma^\gamma \partial u\|_{L^\infty_TL^\infty_\Omega}\lesssim
  \varepsilon+ \mathcal B(\varepsilon,t) \en_K^2(T)\log^2(e+T).
\end{equation}

Finally, we combine \eqref{eq.e0K} and \eqref{Concl02}
to obtain
\begin{align*}
\en_K(T)
\lesssim & {} \ve+ (\ve+\en_K(T))\times\\
&\times \left( \en_K(T)\log^{1/2}(e+T)+\en_K^2(T)\log^2(e+T)+\en_K^4(T)\log^{4}(e+T) \right).
 %+\en_K^5(T)\log^{9/5}(e+T).
\end{align*}
In order to find
\begin{align*}
\en_K(T)\le C_1 \ve+{\mathcal R}\left(\en^P_K(T)\log^Q(e+T)\right)
(\varepsilon+\en_K(T))
\end{align*}
with as larger $P/Q$ as possible, we take
$$
\mathcal R(\tau):=C_2 (\tau+\tau^2+\tau^4)  %+\tau^5)
$$
and $P=Q=1$.
Recalling the discussion in Section~\ref{AAPMT}, we obtain Theorem~\ref{thm.mainsemi}.

\section{Proof of pointwise estimates}\label{sec.pointwise}
%%%%%%%%%%%%%%%%%%%%%%%%%%%%%%%%%%%%%%%%%%%%%%%%%%%%%%%
In this section, we go back to the Neumann problem~\eqref{eq.PMfT} and
will prove Theorem~\ref{main} by combining the decay estimates for the Cauchy problem in $\R^2$ and
the local energy decay estimate through the cut-off argument.
%%%%%%%%%%%%%%%%%%%%%%%%%%%%%%%%%%%%%%%%%%%%%%%%%%%%%%
\subsection{Decomposition of solutions}
Recall the definitions of  $X(T)$ and $S[\vec{u}_0, f](t,x)$, $K[\vec{u}_0](t,x)$, $L[f](t,x)$
in Subsection~\ref{LinearNeumann}.
In the same manner, the solution of the Cauchy problem
\begin{equation}\label{eq.PCgT}
\begin{array}{ll}
(\partial_t^2-\Delta) v = g & (t,x) \in (0,T)\times {\R}^2,\\
 v(0,x)=v_0(x), & x\in \R^2,\\
(\partial_t v)(0,x)=v_1(x), & x\in {\R}^2,
\end{array}
\end{equation}
will be denoted by $S_0[\vec{v}_0, g](t,x)$ with $\vec{v}_0=(v_0, v_1)$.
Then we have
$$
S_0[\vec{v}_0,g](t,x)= K_0[\vec{v}_0](t,x)+L_0[f](t,x),
$$
where
$K_0[\vec{v}_0](t,x)$ and
$L_0[g](t,x)$ are the solutions of \eqref{eq.PCgT} with $g=0$ and $\vec{v}_{0}=(0,0)$, respectively.
In other words, $K_0[\vec{v}_0](t,x)=S_0[\vec{v}_0, 0](t,x)$ and $L_0[g](t,x)=S_0[(0,0), g](t,x)$.

Now we proceed to introduce the cut-off argument.
For $a >0$, we denote by $\psi_a$ a smooth radially symmetric function
on ${\R}^2$ satisfying
\begin{equation}\label{eq.psia}
\begin{cases}
\psi_a(x)=0, &  |x| \le a, \\
\psi_a(x)=1, &  |x| \ge a+1.
\end{cases}
\end{equation}
% Finally, we set
% $$
% \Omega_r=\Omega \cap B_r(0);
% $$
% being $\mathcal O\subset\subset B_1(0)$, we see that $\Omega_a\not= \emptyset $ for any $a\ge 1$.

\begin{lemma}\label{lemma.decomposition}
Fix $a\ge 1$. Let $(u_0, u_1, f)\in X(T)$.
% $f \in \mathcal C^\infty([0,T)\times \bar\Omega)$ and $\vec{u}_0\in \mathcal C^\infty
% (\bar\Omega)\times \mathcal C\infty(\bar\Omega)$ satisfy
% the compatibility condition of infinity order in $\Omega$ for \eqref{eq.PMfT}.
Assume that for any $t\in (0,T)$ one has
$$
\text{supp}\,f(t,\cdot) \subset \overline{\Omega_{t+a}}\quad \text{and}\quad
\text{supp}\,u_0 \subset \overline{\Omega_a}, \
\text{supp}\,u_1 \subset \overline{\Omega_a}.
$$
Then we have
\begin{eqnarray}\label{eq.omo}
S[\vec{u}_0, f](t,x)=\psi_a(x) S_0[ \psi_{2a} \vec{u}_0, \psi_{2a}f](t,x)+\sum_{i=1}^4 S_i[ \vec{u}_0, f](t,x),
\end{eqnarray}
where
\begin{eqnarray}\label{eq.S1}
&& S_1[\vec{u}_0,f](t,x)=(1-\psi_{2a}(x))L[\,[\psi_a,-\Delta]S_0[\psi_{2a} \vec{u}_0, \psi_{2a}f] ](t,x),
\\ \label{eq.S2}
&& S_2[\vec{u}_0,f](t,x)=-L_0[\,[\psi_{2a},-\Delta]L[\,[\psi_a,-\Delta]S_0[\psi_{2a} \vec{u}_0, \psi_{2a}f] ] ](t,x),
\\ \label{eq.S3}
&& S_3[\vec{u}_0, f](t,x)=(1-\psi_{3a} (x)) S[(1-\psi_{2a}) \vec{u}_0, (1-\psi_{2a})f](t,x),
\\ \label{eq.S4}
&& S_4[\vec{u}_0, f](t,x)=-L_0[\,[\psi_{3a},-\Delta] S[(1-\psi_{2a}) \vec{u}_0, (1-\psi_{2a})f] ](t,x).
\end{eqnarray}
\end{lemma}

For the proof, we refer to \cite{Kub06}.

Observe that the first term on the right-hands side of
(\ref{eq.omo}) can be evaluated by applying the
decay estimates for the whole space case. In contrast,
the local energy decay estimates for the mixed problem
work well in estimating $S_j[\vec{u}_0, f]$ for $1\le j\le 4$, because
we always have some localized factor in front of the operators $L$, $S$ and
in their arguments.
%%%%%%%%%%%%%%%%%%%%%%%%%%%%%%%%%%%%%%%%%%%%%%%%%%%%%%%
\subsection{Known estimates for the $2$D linear Cauchy problem}
In this subsection we recall the decay estimates for solutions of homogeneous wave equation.
Since $\Lambda K_0[v_0,v_1]=K_0[\Lambda v_0, \Lambda v_1]$ by
\eqref{eq.commute}, we find that Proposition 2.1 of \cite{k93} leads to the following.

\begin{lemma}
Let $m\in\N$.
For any $(v_0,v_1) \in \mathcal C^\infty_0(\R^2)\times \mathcal C^\infty_0(\R^2)$, it holds that
\begin{equation}\label{eq.kubota}
\langle t+|x| \rangle^{1/2}\log^{-1}\left(e+ \frac{\langle t+|x|\rangle}{\langle t-|x|\rangle}\right)\sum_{|\beta|\le m}|\Gamma^\beta K_0[v_0,v_1](t,x)|
\lesssim
\mathcal B_{3/2,m}[v_0,v_1].
\end{equation}
Under the same assumption, for any $\mu>0$ we have
\begin{equation}\label{eq.kubota2}
\langle t+|x|\rangle^{1/2}\langle t-|x|\rangle^{1/2}\sum_{|\beta|\le m}|\Gamma^\beta K_0[v_0,v_1](t,x)|\lesssim
\mathcal B_{2+\mu,m}[v_0,v_1].
\end{equation}
\end{lemma}

For $\kappa\ge 1$ and $\tau\ge 0$, we define
$$
\Psi_\kappa(\tau):=\begin{cases}
        1, & \kappa>1, \\
        \log(e+\tau), & \kappa=1.
       \end{cases}
$$
The following two lemmas are proved for $m=0$ in \cite{DiF03}.
For the general case, see \cite{KP}.
\begin{lemma}
Let $\kappa \ge 1$ and $m\in \N$. Then we have
\begin{eqnarray}\label{eq.decayMGbis}
\sum_{|\delta|\le m}|\Gamma^\delta L_0[g](t,x)|
\lesssim \Psi_{\kappa}(t+|x|)\sum_{|\delta|\le m}\|\langle y\rangle^{1/2}W_{1/2,\kappa}(s,y) \Gamma^\delta g(s,y)\|_{L^\infty_tL^\infty},
\end{eqnarray}
and
\begin{eqnarray}
\label{eq.decayMG}
&&\langle t+|x| \rangle^{1/2}\log^{-1}\left(e+ \frac{\langle t+|x|\rangle}{\langle t-|x|\rangle}\right)
\sum_{|\delta|\le m}|\Gamma^\delta L_0[g](t,x)|\lesssim\nonumber
\\&&\quad
\lesssim \Psi_{\kappa}(t+|x|) \sum_{|\delta|\le m}\|\langle y\rangle^{1/2}W_{1,\kappa}(s,y) \Gamma^\delta g(s,y)\|_{L^\infty_tL^\infty}
\end{eqnarray}
for any $(t,x)\in [0,T)\times \R^2$.
\end{lemma}

\begin{lemma}
Let $0<\sigma<3/2$, $\kappa>1$, $\mu\ge 0$, $0<\eta<1$ and $m\in \N$. Then, for any $(t,x)\in [0,T)\times \R^2$, one has
\begin{eqnarray}
&&\sum_{|\delta|\le m}|\Gamma^\delta \partial L_0[g](t,x)|
\lesssim
\nonumber\\
&&\quad \lesssim w_{\sigma}(t,x) \Psi_{\mu+1}(t+|x|)
\sum_{|\delta|\le m+1}\|\langle y\rangle^{1/2+\kappa}\langle s+
|y| \rangle^{\sigma+\mu} \Gamma^\delta g(s,y)\|_{L^\infty_tL^\infty},
\label{eq.decayMG2}
\\
&&\sum_{|\delta|\le m}|\Gamma^\delta \partial L_0[g](t,x)|
\lesssim \nonumber\\
&&\quad \lesssim w_{1-\eta}(t,x)\log(e+t+|x|)  %\log(e+t)
\sum_{|\delta|\le m+1}\|\langle y\rangle^{1/2}W_{1,1}(s,y) \Gamma^\delta g(s,y)\|_{L^\infty_tL^\infty}.
\label{eq.decayMG4}
\end{eqnarray}
\end{lemma}

\subsection{The local energy decay estimates}

%%%%%%%%%%%%%%%%%%%%%%%%%%%%%%%%%%%%%%%%%%%%%%%%%%%%%%%
We come back to the linear problem \eqref{eq.PMfT}.
Let $X_a(T)$ be the set of all $({u}_0, {u}_1,f)\in X(T)$ such that
\begin{align}\label{eq.Xa}
& {u}_0(x)={u}_1(x)=0 \text{ for } |x|\ge a, \\
& f(t,x)=0  \text{ for } |x|\ge a, \; t\in [0,T).\label{eq.Xaa}
\end{align}

The following local energy decay will be used in the proof of the
pointwise estimate. 

\begin{lemma}\label{LocalEnergyDecay}
Assume that ${\mathcal O}$ is convex.
Let $a, b>1$, $\gamma\in (0, 1]$
 and $m \in {\mathbb N}$.
If $\Xi=({u}_0, {u}_1, f) \in X_{a}(T)$, then for any $t\in [0,T)$ one has
\begin{eqnarray}
&&  \sum_{|\alpha| \le m} \jb{t}^\gamma \| \partial^\alpha
 S[\Xi](t)\|_{L^2(\Omega_b)}\lesssim
\nonumber\\
&&\quad \lesssim \| {u}_0 \|_{{H}^{m}(\Omega)}
 +\| {u}_1 \|_{{H}^{m-1}(\Omega)}
 +\log(e+t)\sum_{|\alpha|\le m-1}\| \jb{s}^{\gamma} (\pa^\alpha f)(s,y) \|_{L^\infty_t L^2_\Omega}.
\label{eq.LE}
\end{eqnarray}
\end{lemma}
\noindent{\it Proof.}
For $a,b>1$, it is known that there exists a
positive constant $C=C(a,b)$ such that
\begin{align}
\int_{\Omega_b}(|\partial_t K[\vec \phi_0](t,x)|^2+
|\nabla_x K[\vec \phi_0](t,x)|^2+&
|K[\vec \phi_0](t,x)|^2) \,dx \lesssim \nonumber\\
&
\lesssim \jb{t}^{-2}\left(\|\phi_0\|^2_{H^1(\Omega)}+ \|\phi_1\|^2_{L^2(\Omega)}\right)
\label{obstacle}
\end{align}
for any $\vec \phi_0=(\phi_0, \phi_1)\in H^{2}(\Omega)\times H^1(\Omega)$
satisfying $\phi_0(x)=\phi_1(x)\equiv 0$ for $|x|\ge a$ and satisfying also the compatibility condition of order $0$,
that is to say, $\pa_\nu \phi_0(x)=0$ for $x\in \pa \Omega$
(see for instance Lemma~2.1 of \cite{SeSh03}; see also Morawetz \cite{Mor75} and Vainberg \cite{Vai75}).

Now let $({u}_0, {u}_1, f)\in X_{a}(T)$ with some $a>1$.
Let $u_j$ for $j\ge 2$ be defined as in Definition~\ref{def.complinear}.
Then, by Duhamel's principle, it follows that
\begin{align}
& \partial_t^j S[({u}_0, {u}_1,f)](t,x)
\nonumber\\
& \qquad =K[({u}_j, {u}_{j+1})](t,x)+
\int_0^t K\bigl[\bigl(0,(\partial_t^j f)(s) \bigr) \bigr](t-s,x) ds
\label{DP}
\end{align}
for any nonnegative integer $j\in \N^*$ and any $(t,x) \in [0,T) \times \Omega$.
Observe that $(u_j, u_{j+1}, 0)$
satisfies the compatibility condition of order $0$, because $(u_0, u_1, f)\in X(T)$ implies
$\pa_\nu u_j=0$ on $\pa \Omega$; the compatibility condition of order $0$ is
also trivially satisfied for $\bigl(0,(\pa_s^j f)(s),0\bigr)$ for all $s\ge 0$.

\noindent
Therefore, by (\ref{obstacle}) we have
\begin{eqnarray*}
 \sum_{|\alpha|\le 1}
 \|\partial^\alpha K[{u}_j, {u}_{j+1}](t)\|_{L^2({\Omega_b})}
& \lesssim &  \jb{t}^{-1} \left(\|{u}_j\|_{H^1(\Omega)}+\|{u}_{j+1}\|_{L^2(\Omega)}\right)
\\
& \lesssim & \jb{t}^{-1} \bigl(\|{u}_0\|_{H^{j+1}(\Omega)}+\|{u}_1\|_{H^{j}(\Omega)}
+\sum_{k=0}^{j-1} \|(\partial_t^k f)(0)\|_{L^2(\Omega)}\bigr)
\end{eqnarray*}
and
\begin{eqnarray*}
 \sum_{|\alpha|\le 1}\int_0^t \|\partial^\alpha K[(0,(\partial_t^j f)(s))](t-s)\|_{L^2({\Omega_b})} ds
 &\lesssim &\int_0^t  \jb{t-s}^{-1}
\,\|(\partial_t^j f)(s)\|_{L^2(\Omega)} ds
\\
\qquad &\lesssim &\jb{t}^{-\gamma} {\rm log}(e+t) \sup_{0\le s \le t} \jb{s}^\gamma
  \|(\partial_t^j f)(s)\|_{L^2(\Omega)}
\end{eqnarray*}
for any $\gamma \in (0,1]$.
In conclusion for any  $j\in\N^*$, we have
\begin{eqnarray}
&& \sum_{|\alpha|\le 1}\| \partial^\alpha \partial^j_{t} S[({u}_0, {u}_1, f)](t)\|_{L^2(\Omega_b)}\lesssim
\nonumber\\
&& \quad \lesssim \jb{t}^{-\gamma} \bigl(
\|{u}_0\|_{H^{j+1}(\Omega)}+\|{u}_1\|_{H^{j}(\Omega)}
+\sum_{k=0}^j\log(e+t) \sup_{0\le s \le t} \jb{s}^\gamma\|(\partial_t^k f)(s)\|_{L^2(\Omega)}\bigr).
\label{LE1}
\end{eqnarray}

In order to evaluate $\partial^\alpha S[\Xi]$ for $2\le |\alpha| \le m$,
we have only to combine (\ref{LE1}) with a variant of (\ref{eq.ap1})\,:
\begin{equation}\label{LE2}
 \|\varphi\|_{H^m(\Omega_b)} \lesssim \|\Delta_x \varphi\|_{H^{m-2}(\Omega_{b^\prime})}+\|\varphi\|_{H^{m-1}(\Omega_{b^\prime})},
\end{equation}
where $1<b<b^\prime$ and $\varphi \in H^m(\Omega)$ with $m \ge 2$;
we can easily obtain \eqref{LE2} from \eqref{eq.ap1} by cutting off $\varphi$ for $|x|\ge b'$.

In order to complete the proof, one has to apply this inequality recalling the equation $\Delta S[\Xi]=\partial_t^2 S[\Xi]-f$.
Invoking \eqref{LE1}, we finally get the basic estimate \eqref{obstacle}.
\hfill$\qed$

\subsection{Proof of Theorem~\ref{main}}
The following lemma is the main tool for the proof of Theorem~\ref{main}.
\begin{lemma}\label{KataLem}
Let ${\mathcal O}$ be a convex set.
Let $a,b>1$, $0<\rho\le 1$, $m\in \N^*$ and $\kappa \ge 1$.

\noindent
{\rm (i)} Suppose that $\chi$ is a smooth function on $\R^2$ satisfying ${\rm supp}\, \chi \subset  B_b$.
If $\Xi=({u}_0, {u}_1,f) \in X_{a}(T)$,
then
\begin{eqnarray}
&& \langle t \rangle^\rho
   \sum_{|\delta|\le m}|\Gamma^\delta (\chi S[ \Xi  ])(t,x)|\lesssim
\nonumber\\
&&\lesssim \| {u}_0 \|_{{H}^{m+2}(\Omega)}+\| {u}_1 \|_{{H}^{m+1}(\Omega)}
 + \log(e+t) \sum_{|\beta|\le m+1} \|\langle s\rangle^{\rho} \partial^\beta f(s,y)\|_{L^\infty_tL^\infty({\Omega_a})}
\label{KataL01}
\end{eqnarray}
for $(t,x)\in[0, T)\times \overline{\Omega}$.

\medskip

\noindent
{\rm (ii)}
Let $g\in  \mathcal C^{\infty}([0,T)\times \R^2)$  such that
$\supp g (t,\cdot)\subset \overline{B_a\setminus B_1}$
for any $t\in [0,T)$.
Then
\begin{eqnarray}
\label{KataL02}
 \sum_{|\delta|\le m}|\Gamma^\delta
L_0[g](t,x)| %\lesssim \\ &&
\lesssim
 \sum_{|\beta|\le m}  \|\langle s\rangle^{1/2} \partial^\beta g(s,y)\|_{L^\infty_tL^\infty({\Omega_a})},
\end{eqnarray}
and for any $0\le \eta<\rho$ we have
\begin{align}
\label{KataL03}
& w^{-1}_{\rho-\eta}(t,x)
\sum_{|\delta|\le m}
|\Gamma^\delta \partial
L_0[g](t,x)|
\lesssim
  \Psi_{\eta+1}(t+|x|)
\sum_{|\beta|\le m+1} \|\langle s\rangle^{\rho} \partial^\beta g(s,y)\|_{L^\infty_tL^\infty({\Omega_a})}.
\end{align}
for $(t,x)\in [0,T)\times \overline{\Omega}$.

\medskip

\noindent
{\rm (iii)} Let $(v_0,v_1,g)\in \mathcal C^{\infty}(\R^2)\times \mathcal C^{\infty}(\R^2)\times \mathcal C^{\infty}([0,T)\times \R^2)$.
If $v_0=v_1=g(t,\cdot)=0$ for any $x\in B_1$ and $t\in [0,T)$, then
\begin{eqnarray}
&& \langle t \rangle^{1/2} \sum_{|\beta|\le m}
   |\Gamma^\beta S_0[v_0,v_1, g ](t,x)|
\lesssim
\nonumber\\
&&\quad \lesssim {\mathcal A}_{3/2, m}[v_0,v_1]
 +\Psi_\kappa(t+|x|) \sum_{|\beta|\le m} \|\langle y\rangle^{1/2} W_{1,\kappa}(s,y)\Gamma^\beta g(s,y)\|_{L^\infty_tL^\infty({\Omega})}
\label{KataL04}
\end{eqnarray}
for $(t,x)\in [0,T)\times {\overline{\Omega}_b}$.
\end{lemma}

\noindent{\it Proof.}\ \
First we note that for any smooth function $h:[0,T)\times \overline{\Omega}\to \R$
such that ${\rm supp}\, h(t, \cdot)\subset B_R$ for any $t\in [0,T)$ and suitable $R>1$,
 it holds that
\begin{equation}
\sum_{|\beta|\le m}|\Gamma^\beta h(t,x)|\lesssim \sum_{|\beta|\le m} |\partial^\beta h(t,x)|.
\label{KataM01}
\end{equation}
Clearly the same estimate holds for $h:[0,T)\times \R^2\to \R$.

We start with the proof of \eqref{KataL01}.
Let $\Xi \in X_{a}(T)$ and $0<\rho\le 1$.
For $(t,x)\in [0, T)\times \overline{\Omega}$,
combining \eqref{KataM01}
with the standard Sobolev inequality and then applying the local energy decay \eqref{eq.LE}, we get
\begin{eqnarray*}
&&\langle t \rangle^\rho
\sum_{|\beta|\le m}|\Gamma^{\beta}(\chi S[ \Xi  ])  (t,x)|
  \lesssim \langle t\rangle^\rho \!\!\! \sum_{|\beta|\le m+2}
        \|{\partial^\beta S[ \Xi ](t)}\|_{L^2(\Omega_b)}
\\
&& \quad \lesssim
\| {u}_0 \|_{{H}^{m+2}(\Omega)}+\| {u}_1 \|_{{H}^{m+1}(\Omega)}
 +\log(e+t) \sum_{|\beta|\le m+1} \| \jb{s}^{\rho} \pa^\beta f(s,y) \|_{L^\infty_t L^2_\Omega}.
\end{eqnarray*}
Since $\text{supp}\,f(t,\cdot) \subset \overline{\Omega}_{a}$
implies $\|{\partial^\beta f(s)}\|_{L^2(\Omega)}\lesssim \|{\partial^\beta f(s)}\|_{L^\infty({\Omega}_a)}$,
we obtain \eqref{KataL01}.

Next we prove \eqref{KataL02} by the aid of the decay estimates for the linear Cauchy problem.
By
\eqref{eq.decayMGbis} for some $\kappa>1$, we find
\begin{eqnarray*}
 \sum_{|\delta|\le m}|\Gamma^\delta %S_0[(0,0, g)]
 L_0[g](t,x)|
\lesssim
 \sum_{|\delta|\le m} \|\langle y\rangle^{1/2}W_{1/2,\kappa}(s,y)\Gamma^\delta g(s,y)\|_{L^\infty_tL^\infty}.
\end{eqnarray*}
Using
the assumption
${\rm supp}\, g(t,\cdot) \subset \overline{B_a\setminus B_1}\subset {\overline{\Omega}_{a}}$, we gain \eqref{KataL02}.

Similarly, if we use
\eqref{eq.decayMG2} (with $\sigma $ being replaced by $\rho-\eta$ and $\mu$ by $\eta$), instead of \eqref{eq.decayMGbis}, then we get \eqref{KataL03}.

Finally we prove \eqref{KataL04} by using \eqref{eq.kubota} and \eqref{eq.decayMG}.
It follows that
\begin{eqnarray*}
&& \langle t+|x| \rangle^{1/2} \log\left(e+\frac{\langle t+|x|\rangle }{\langle t-|x|\rangle }\right)
\sum_{|\beta|\le m} |\Gamma^\beta S_0[\vec{v}_0, g ](t,x)|\lesssim
\\&& \quad
\lesssim {\mathcal B}_{3/2, m}[\vec{v}_0]
+ \Psi_\kappa(t+|x|) \sum_{|\beta|\le m} \|\langle y\rangle^{1/2}W_{1,\kappa}(s,y) \Gamma^\beta g(s,y)\|_{L^\infty_tL^\infty}
\end{eqnarray*}
for $(t,x)\in [0, T)\times \R^2$.
Observe that the logarithmic term on the left-hand side is equivalent to a constant
when $x \in \overline{\Omega_b}$.
Thus we get \eqref{KataL04}, because our assumption ensures that support of data and ${\rm supp}\,g(t,\cdot)$ are contained in $\Omega$.
This completes the proof.\hfill$\qed$

\bigskip
Now we are in a position to prove Theorem~\ref{main}.
\begin{proof}[Proof of Theorem~$\ref{main}$]
%%%%%%%%%%%%%%%%%%%%%%%%%%%%%

%\begin{corollary}
%Proposition \ref{prop.pweIII}, Proposition \ref{prop.pwe1} hold
%\end{corollary}

%\begin{proof}
%The statement (i) directly implies Proposition \ref{prop.pwe1}.
%The statements (ii) and (iii)  directly implies Proposition \ref{prop.pwe1}.
%\end{proof}

%\noindent{\it Proof}  of Theorem \ref{main}.}\ \
%\noindent{\it Proof.}\ \
According to Lemma \ref{lemma.decomposition} with $a=1$, we can write
\begin{equation}\label{decomposition}
S[{\Xi}](t,x)=\psi_1(x) S_0[\psi_2 \Xi](t,x)
{}+\sum_{i=1}^4 S_i[\Xi](t,x)
\end{equation}
for $(t,x)\in [0,T)\times {\overline{\Omega}}$,
where $\psi_a$ is defined by (\ref{eq.psia}) and
$S_i[\Xi]$ for $1\le i\le 4$ are defined by \eqref{eq.S1}--\eqref{eq.S4} with $a=1$.
It is easy to check that
\begin{equation}\label{eq.compsi}
[\psi_a,-\Delta]h(t,x)=
   h(t,x) \Delta \psi_a(x)+2\nabla_{\!x}\,h(t,x) \cdot \nabla_{\!x}\, \psi_a(x)
\end{equation}
for $(t, x) \in [0,T)\times {\overline{\Omega}}$, $a \ge 1$ and any smooth function $h$.
Note that this identity implies
\begin{equation}\label{eq.comX}
(0,0, [\psi_a, -\Delta]h)\in X_{a+1}(T)
\end{equation}
because
${\rm supp}\, \nabla_x \psi_a\cup {\rm supp}\, \Delta \psi_a\subset
\overline{B_{a+1}\setminus B_a}$.
%\footnote{I erase $a \ge 1$ and any smooth function $h$ since yet written}

\smallskip
%%%%%%%%%%%%%%%%%%%%%%%
First we prove \eqref{ba3}.
Applying \eqref{eq.kubota} and \eqref{eq.decayMG}, we have
\begin{eqnarray*}
&& \jb{t+|x|}^{1/2}
\log^{-1}\left(e+\frac{\langle t+|x|\rangle}{\langle t-|x|\rangle}\right)
\sum_{|\delta|\le k}\left|\Gamma^\delta S_0[\psi_2\Xi](t,x)\right|\lesssim
\\
&& \quad \lesssim {\mathcal B}_{3/2,k}[\psi_2\vec{u}_0]
 + \sum_{|\delta|\le k}\|\langle y\rangle^{1/2}W_{1,1+\mu}(s,y)\Gamma^\delta (\psi_2 f)(s,y)\|_{L^\infty_tL^\infty}
\\
&& \quad \lesssim  {\mathcal A}_{3/2,k}[\vec{{u}}_0]
 +\sum_{|\delta|\le k} \| |y|^{1/2} W_{1,1+\mu}(s,y) \Gamma^\delta f(s,y)\|_{L^\infty_tL^\infty_{\Omega}},
\end{eqnarray*}
so that
\begin{align}
& \jb{t+|x|}^{1/2}\log^{-1}\left(e+\frac{\jb{t+|x|}}{\jb{t-|x|}}\right)
\sum_{|\delta|\le k}\left| \Gamma^{\delta}\bigl(\psi_1(x) S_0[\psi_2\Xi](t,x)\bigr)\right|
 \lesssim \nonumber\\
& \qquad\qquad \lesssim
{\mathcal A}_{3/2,k}[\vec{u}_0]
 + \sum_{|\delta|\le k}\| |y|^{1/2} W_{1,1+\mu}(s,y) \Gamma^\delta f(s,y)\|_{L^\infty_tL^\infty_{\Omega}}.
  \label{eq.finalS0}
\end{align}

Now we write
$$
S_1[\Xi]=(1-\psi_2)L[[\psi_1,-\Delta]K_0[\psi_2 \vec{u}_0]]+(1-\psi_2)L[[\psi_1,-\Delta]L_0[\psi_2 f]]=: S_{1,1}[\Xi]+S_{1,2}[\Xi].
$$
We can apply \eqref{KataL01} to estimate $S_{1,2}[\Xi]$,
because we have $L[h]=S[0,0,h]$ and ${\rm supp}(1-\psi_2)\subset B_3$ and
because \eqref{eq.comX} guarantees $(0,0,[\psi_1,-\Delta]L_0[\psi_2f]) \in X_{2}$.
Therefore we get
\begin{eqnarray*}
\langle t \rangle^{1/2}\sum_{|\delta|\le k} |\Gamma^\delta S_{1,2}[\Xi](t,x)|&\lesssim&
\log(e+t)\sum_{|\beta|\le k+ 1}
 \bigl\|\langle s\rangle^{1/2} \partial^\beta \bigl([\psi_1,-\Delta]L_0[\psi_2 f] \bigr)(s,x)\bigr\|_{L^\infty_tL^\infty(\Omega_2)}\\
&\lesssim &
\log(e+t)\sum_{|\beta|\le k+2}
 \|\langle s\rangle^{1/2} \partial^\beta L_0[\psi_2 f](s,x)\|_{L^\infty_tL^\infty(\Omega_2)},
\end{eqnarray*}
where we have used \eqref{eq.compsi} to obtain the second line.
Recalling that $L_0[h]=S_0[0,0,h]$ and noting that $\psi_2 f(t,x)=0$ if $|x|\le 2$, we can use \eqref{KataL04} to obtain
\begin{equation}\label{eq.stimaS12}
\langle t \rangle^{1/2} \sum_{|\delta|\le k}|\Gamma^\delta S_{1,2}[\Xi](t,x)| \lesssim \log(e+t)
 \sum_{|\beta|\le k+2}\| |y|^{1/2} W_{1,1+\mu}(s,y)\Gamma^\beta f(s,y)\|_{L^\infty_tL^\infty_{\Omega}}
\end{equation}
for $(t,x)\in [0,T)\times {\overline{\Omega}}$.

In order to estimate $S_{1,1}[\Xi]$, we combine the Sobolev embedding and
the local energy decay estimate \eqref{eq.LE} with $\gamma=1$. Then we get
\begin{eqnarray*}
\sum_{|\delta|\le k}|\Gamma^\delta S_{1,1}[\Xi](t,x)|
&\lesssim&
\|(1-\psi_2) L[[\psi_1,-\Delta]K_0[\psi_2 \vec{u}_0]](t,\cdot)\|_{H^{2+k}(\Omega)}\\
&\lesssim&
\|S[0,0,[\psi_1,-\Delta]K_0[\psi_2 \vec{u}_0]](t,\cdot)\|_{H^{2+k}(\Omega_3)}
\\
&\lesssim& \langle t \rangle^{-1} \log(e+t)
\sum_{|\delta|\le k+1}
\|\langle s\rangle\pa^\delta\bigl( [\psi_1,-\Delta] K_0[\psi_2 \vec{u}_0]\bigr)(s,y) \|_{L^\infty_{t} L^2_{\Omega} }
\\
&\lesssim&
\langle t \rangle^{-1} \log(e+t)\sum_{|\beta|\le k+2}\| \langle s\rangle \partial^\beta K_0 [\psi_2
\vec{u}_0](s,y) \|_{L^\infty_{t} L^\infty(\Omega_2)}.
\end{eqnarray*}
Then we use \eqref{eq.kubota2}; recalling that we are in a bounded $y$-domain, for any $\mu>0$ we get
\begin{equation}\label{eq.stimaS11}
\jb{t}^{1/2} \jb{t+|x|}^{1/2} \log^{-1} (e+t)
 \sum_{|\delta|\le k}|\Gamma^\delta S_{1,1}[\Xi](t,x)|\lesssim \mathcal B_{2+\mu,2+k}[\psi_2\vec{u}_0]
\lesssim  \mathcal A_{2+\mu,2+k}[\vec{u}_0]
\end{equation}
for any $(t,x)\in [0,T)\times {\overline{\Omega}}$.

Now we proceed estimating $S_3[\Xi]$. Because $(1-\psi_2)\Xi\in X_{3}(T)$
for any $\Xi\in X(T)$, taking $\rho=1-\mu$ in
\eqref{KataL01} we get
\begin{eqnarray}
&& \langle t\rangle^{1/2} \sum_{|\delta|\le k}|\Gamma^\delta  S_3[\Xi](t,x)|\lesssim
\label{eq.stimaS3}\\
&&\quad  \lesssim \langle t\rangle^{-1/2+\mu}
\Big(
\|{u}_0\|_{H^{k+2}(\Omega_{3})}+\|{u}_1\|_{H^{k+1}(\Omega_{3})}+
\log(e+t) \sum_{|\beta|\le k+1} \|\langle s\rangle^{1-\mu} \partial^\beta f(s,y)\|_{L^\infty_tL^\infty(\Omega_3)}
 \Big)
\nonumber
\end{eqnarray}
for $(t,x)\in [0,T)\times {\overline{\Omega}}$.

By using the trivial inequality $\langle s\rangle^{1-\mu}\lesssim |y|^{1/2} W_{1, 1}(s,y)$ in $[0,T)\times\Omega_3$, from
\eqref{eq.stimaS12}, \eqref{eq.stimaS11} and \eqref{eq.stimaS3} we can conclude that
\begin{align}
&\langle t\rangle^{1/2}\sum_{|\delta|\le k}|\Gamma^\delta S_1[\Xi]|+\langle t\rangle^{1/2}\sum_{|\delta|\le k}|\Gamma^\delta S_3[\Xi]|\lesssim
\nonumber\\
& \lesssim
\jb{t}^{-(1/2)+\mu}
 \mathcal A_{2+\mu,2+k}[\vec{u}_0]+
\log(e+t)\sum_{|\beta|\le 2+k}\||y|^{1/2}W_{1,1+\mu}(s,y)\Gamma^\beta f(s,y)\|_{L^\infty_tL^\infty_{\Omega}}. \label{eq.finalS1S3log}
\end{align}

Finally we consider the terms $S_2[\Xi]$, $S_4[\Xi]$. Let us set $g_j[\Xi]=(\partial_t^2-\Delta) S_j[\Xi]$ for $j=2, 4$.
Recalling the definition of $L_0$, we  find
\begin{eqnarray*}
g_2[\Xi]&=& -[\psi_2,-\Delta]  L\bigl[\,[\psi_1,-\Delta]
             S_0[\psi_2 \Xi]\bigr];\\
g_4[\Xi]&=& -[\psi_3,-\Delta] S[(1-\psi_2)\Xi].
\end{eqnarray*}
Having in mind \eqref{eq.compsi} we can say that $g_2$ and $g_4$ have the same structures as $S_1$ and $S_3$,
but they contain one more derivative.
Therefore, arguing similarly to the derivation of \eqref{eq.finalS1S3log}, we arrive at
\begin{align}
&\langle t\rangle^{1/2}\sum_{|\delta|\le k}|\Gamma^\delta g_2[\Xi]|+\langle t\rangle^{1/2}\sum_{|\delta|\le k}|\Gamma^\delta g_4[\Xi]|\lesssim
\nonumber\\
&\lesssim
\jb{t}^{-(1/2)+\mu} \mathcal A_{2+\mu,3+k}[\vec{u}_0]+
\log(e+t)\sum_{|\beta|\le 3+k}\| |y|^{1/2} W_{1,1+\mu}(s,y)\Gamma^\beta f(s,y)\|_{L^\infty_tL^\infty_{{\Omega}}}.
\label{Star01}
\end{align}
On the other hand, we have $S_i[\Xi]= L_0[g_i]$ for $i=2,4$.
Thus, since $g_2$ and $g_4$ are supported on
$\overline{B_4\setminus B_2}$,
we are in a position to apply
\eqref{KataL02}
and we get
\begin{eqnarray}
&& \sum_{|\delta|\le k}\left(|\Gamma^\delta S_2[\Xi]|+|\Gamma^\delta S_4[\Xi]|\right)(t,x)
\lesssim\nonumber\\
&&\quad
\lesssim \mathcal A_{2+\mu,3+k}[\vec{u}_0]+\log(e+t)
\sum_{|\beta|\le 3+k}\| |y|^{1/2} W_{1,1+\mu}(s,y)\Gamma^\beta f(s,y)\|_{L^\infty_tL^\infty_\Omega}.
\label{eq.finalS2S4}
\end{eqnarray}
Now \eqref{ba3} follows from \eqref{eq.finalS0}, \eqref{eq.finalS1S3log} and \eqref{eq.finalS2S4}.

%%%%%%%%%%%%%%%%%%%%%%%
Next we prove \eqref{ba4}.
Trivially one has
\begin{eqnarray*}
&& \sum_{|\delta|\le k} |\Gamma^\delta \partial (\psi_1(x) S_0[\psi_2 \Xi](t,x))|\lesssim
\\
&& \quad \lesssim
\sum_{|\delta|\le k} |\Gamma^\delta \partial S_0[\psi_2 \Xi](t,x)|
+\sum_{|\delta|\le k} |\Gamma^\delta \nabla_x \psi_1(x)| |\Gamma^{\delta} S_0[\psi_2 \Xi](t,x)|.
\end{eqnarray*}
Since in $\Omega$
one has $|y|\simeq \langle y\rangle$, by
\eqref{eq.kubota2} and \eqref{eq.decayMG4} with $\eta=1/2$, we see that
\begin{eqnarray*}
&&\sum_{|\delta|\le k} |\Gamma^\delta \partial S_0[\psi_2 \Xi](t,x)|
  \lesssim \jb{t+|x|}^{-1/2}\jb{t-|x|}^{-1/2} {\mathcal A}_{2+\mu,k+1}[\vec{u}_0]+ \\
&& \quad +w_{1/2}(t,x) \log (e+t+|x|)
\sum_{|\delta|\le k+1}\|
|y|^{1/2}W_{1,1}(s,y) \Gamma^\delta f(s,y)\|_{L^\infty_tL^\infty_{\Omega}}.
\end{eqnarray*}
On the other hand, by \eqref{eq.kubota} and \eqref{eq.decayMG} with $\kappa=1$, we have
\begin{eqnarray*}
&& \jb{t+|x|}^{1/2}
\log^{-1}\left(e+\frac{\langle t+|x|\rangle}{\langle t-|x|\rangle}\right)
\sum_{|\delta|\le k}\left|\Gamma^\delta S_0[\psi_2\Xi](t,x)\right|\lesssim
\\
&& \quad \lesssim  {\mathcal A}_{3/2,k}[\vec{u}_0]+\log (e+t+|x|)
  \sum_{|\delta|\le k}\| |y|^{1/2}W_{1,1}(s,y) \Gamma^\delta f(s,y)\|_{L^\infty_tL^\infty_{\Omega}}.
\end{eqnarray*}
Since the logarithmic term on the left-hand side does not appear when $x \in \Omega_{2}$,
we get
\begin{eqnarray}
 && w_{1/2}^{-1}(t,x) \sum_{|\delta|\le k}\left| \Gamma^{\delta} \partial
   \bigl(\psi_1(x) S_0[\psi_2\Xi]\bigr)(t,x)\right|
 \nonumber\\
&& \quad \lesssim
{\mathcal A}_{2+\mu,k+1}[\vec{u}_0]+\log (e+t+|x|)
  \sum_{|\delta|\le k+1} \| |y|^{1/2}W_{1,1}(s,y) \Gamma^{\delta} f(s,y)\|_{L^\infty_tL^\infty_{\Omega}}.   %\noindent
\label{eq.finalS0bis}
\end{eqnarray}
Therefore, $\partial(\psi_1S_0[\psi_2\Xi])$ has the desired bound.

Let us recall that $|x|$ is bounded in $\supp S_1[\Xi](t,\cdot)\cup \supp S_3[\Xi](t,\cdot)$. In particular we get
$ w_{1/2}^{-1}(t,x)\lesssim \langle t\rangle^{1/2}$. From \eqref{eq.finalS1S3log} we deduce
\begin{eqnarray}
&& \sum_{|\delta|\le k} w_{1/2}^{-1}(t,x)
  \left( |\Gamma^\delta \partial S_1[\Xi](t,x)|+|\Gamma^\delta \partial S_3[\Xi](t,x)| \right)
\lesssim
\nonumber \\
&& \quad \lesssim \mathcal A_{2+\mu,3+k}[\vec{u}_0]+
\log(e+t)
\sum_{|\beta|\le 3+k}\| |y|^{1/2}W_{1,1+\mu}(s,y)\Gamma^\beta f(s,y)
\|_{L^\infty_tL^\infty_{\Omega}}.
\label{eq.finalS1S3logder}
\end{eqnarray}
As for $S_4[\Xi]$, we use a similar estimate to \eqref{eq.stimaS3} 
with $k$ replaced by $k+1$, that is
\begin{align}
&
\langle t\rangle^{1-\mu}
\sum_{|\delta|\le k+1}|\Gamma^\delta g_4[\Xi](t,x)|\lesssim
\nonumber\\
& \lesssim
 \mathcal A_{2+\mu,k+4}[\vec{u}_0]+
\log(e+t)
\sum_{|\beta|\le k+3}\| |y|^{1/2} W_{1,1+\mu}(s,y)\Gamma^\beta f(s,y)\|_{L^\infty_tL^\infty_{\Omega}}.
\label{g2g4}
\end{align}
Applying \eqref{KataL03} with $\rho=1-\mu$ and $\eta=\mu$~($0<\mu\le 1/4$),
we find that
\begin{eqnarray}
&& \sum_{|\delta|\le k}  w_{1-2\mu}^{-1}(t,x)
  |\Gamma^\delta \partial S_4[\Xi]|(t,x)
\lesssim
\nonumber\\
&&\quad
\lesssim \mathcal A_{2+\mu,k+4}[\vec{u}_0]+
\log(e+t)
\sum_{|\beta|\le k+3}\| |y|^{1/2} W_{1,1+\mu}(s,y)\Gamma^\beta f(s,y)\|_{L^\infty_tL^\infty_{\Omega}}.
\label{eq.finalS2S4der}
\end{eqnarray}
For treating $S_2[\Xi]$, we decompose $g_2[\Xi]$ into
$g_{2,1}[\Xi]$ and $g_{2,2}[\Xi]$ as was done for evaluating $S_1[\Xi]$.
Then $L_0[g_{2,1}]$ can be estimated as $S_4[\Xi]$.
On the other hand, using \eqref{KataL03} with $\rho=1/2$ and $\eta=0$ for $L_0[g_{2,2}]$,
we arrive at
\begin{eqnarray}
&& \sum_{|\delta|\le k}  w_{1/2}^{-1}(t,x)
  |\Gamma^\delta \partial S_2[\Xi]|(t,x)
\lesssim
\nonumber\\
&&\quad
\lesssim \mathcal A_{2+\mu,4+k}[\vec{u}_0]+
\log^2(e+t+|x|)
\sum_{|\beta|\le 4+k}\| |y|^{1/2} W_{1,1+\mu}(s,y)\Gamma^\beta f(s,y)\|_{L^\infty_tL^\infty_{\Omega}}.
\label{eq.finalS2S4derBis}
\end{eqnarray}
Thus we obtain \eqref{ba4} from \eqref{eq.finalS0bis}, \eqref{eq.finalS1S3logder},
\eqref{eq.finalS2S4der}, and \eqref{eq.finalS2S4derBis}.

In order to show \eqref{ba4weak}, we remark that
$w_{1/2}\le w_{(1/2)-\eta}$ so that
in \eqref{eq.finalS0bis} we can replace $w_{1/2}$ with $w_{1/2-\eta}$.
Moreover, \eqref{eq.finalS1S3logder}
and \eqref{g2g4} hold with $\mu=0$ if we replace
$\log(e+t)$ by $\log^2(e+t)$, thanks to
\eqref{KataL04} with $\kappa=1$.
Therefore, the application of \eqref{KataL03} with $\rho=1/2$ and $0<\eta<1/2$ leads to \eqref{eq.finalS2S4der} with $w_{1/2}^{-1}$ replaced by $w^{-1}_{(1/2)-\eta}$
and $\mu=0$ in the second term of the right-hand side.
Hence we get \eqref{ba4weak}.

Finally, we prove \eqref{ba4t}. We put $\eta'=\eta/2$.
By \eqref{eq.kubota2} and \eqref{eq.decayMG4}, we see that
\begin{eqnarray*}
&&\sum_{|\delta|\le k +1} |\Gamma^\delta \partial_t (\psi_1(x) S_0[\psi_2 \Xi](t,x))|
 \lesssim \sum_{|\delta|\le k +1} |\Gamma^\delta \pa_t S_0[\psi_2\Xi](t,x)|\lesssim\\
&& \quad \lesssim \jb{t+|x|}^{-1/2}\jb{t-|x|}^{-1/2} {\mathcal A}_{2+\mu,k+ 2}[\vec{u}_0]+
\\
&& \qquad +w_{1-\eta'}(t,x) \log (e+t+|x|)
\sum_{|\delta|\le k+ 2}\| |y|^{1/2}W_{1,1}(s,y) \Gamma^\delta f(s,y)\|_{L^\infty_tL^\infty_{\Omega}}.
\end{eqnarray*}
Therefore, $\pa_t\bigl(\psi_1S_0[\psi_2\Xi] \bigr)$ has the desired bound
because $w_{1-\eta'}\le w_{1-\eta}$.

Combining this estimate with \eqref{KataL01}, we obtain the estimate for $S_1[\Xi]$. Indeed, for $0<\eta<1$ we have
\begin{eqnarray*}
\langle t \rangle^{1-\eta'}
   \sum_{|\delta|\le  k +1}|\Gamma^\delta \partial_t S_1[ \Xi  ](t,x)|\lesssim
 \log(e+t) \sum_{|\beta|\le k+ 2}
\bigl\|\langle s\rangle^{1-\eta'} \partial^\beta \partial_t \bigl([\psi_1, -\Delta]S_0[\psi_2\Xi]
\bigr)(s,y)\bigr\|_{L^\infty_tL^\infty({\Omega_2})}.
\end{eqnarray*}
Recalling \eqref{eq.compsi}, we can use the estimate of
$\partial_t (\psi_1 S_0[\psi_2 \Xi])$
adding two derivatives.
In conclusion, we have
\begin{eqnarray*}
&& \langle t \rangle^{1-\eta'}
\sum_{|\delta|\le k +1} |\Gamma^\delta \partial_t S_1[\Xi](t,x)|\lesssim
\Theta_{\mu, k+4}(t)
\end{eqnarray*}
for $(t,x)\in [0,T)\times {\overline{\Omega}}$,
 where
$$
 \Theta_{\mu, m}(t):=
 {\mathcal A}_{2+\mu, m}[\vec{u}_0]+
{\rm log}^2(e+t)
\sum_{|\delta|\le m}\| |y|^{1/2}W_{1,1}(s,y) \Gamma^\delta f(s,y)\|_{L^\infty_tL^\infty_{\Omega}}.
$$
Since we have $(1-\psi_2)\Xi\in X_{3}(T)$ for any $\Xi\in X(T)$, by using \eqref{KataL01}
with $\rho=1-\eta'$ we have
$$
\jb{t}^{1-\eta'}\sum_{|\delta|\le k+1} |\Gamma^\delta \partial_t S_3[\Xi](t,x)|\lesssim
\Theta_{\mu, k+3}(t).
$$

In order to treat $S_2[\Xi]$ and $S_4[\Xi]$, we set $g_j[\Xi]=(\partial_t^2-\Delta) S_j[\Xi]$ for $j=2, 4$
as before.
Going similar lines to the estimates for $S_1[\Xi]$ and $S_3[\Xi]$, with a derivative more, we can reach at
\begin{eqnarray*}
\langle t\rangle^{1-\eta'} \sum_{|\delta|\le k +1}|\Gamma^\delta \partial_t g_2[\Xi]|+\langle t\rangle^{1-\eta'}
 \sum_{|\delta|\le k +1}|\Gamma^\delta \partial_t g_4[\Xi]|\lesssim \Theta_{\mu, k+5}(t).
\end{eqnarray*}
Let us recall that $g_2$ and $g_4$ are supported on $\overline{B_4\setminus B_2}$ and
$
\pa_t S_i[\Xi]= L_0[{\pa_t}g_i]$ for $i=2,4$.
We are in a position to apply \eqref{KataL03} (with $\rho=1-\eta'$,
and $\eta$ replaced by $\eta'$)
and
obtain
$$
w_{1-\eta}^{-1}(t,x)\sum_{|\delta|\le k} \sum_{i=2, 4} |\Gamma^\delta \pa \pa_t S_i[\Xi](t,x)|
\lesssim \sum_{i=2,4}\sum_{|\delta|\le k+1}\|\jb{s}^{1-\eta'}\pa^\beta\pa_tg_i(s,y)\|_{L^\infty_tL^\infty(\Omega_4)}
\lesssim \Theta_{\mu, k+5}(t).
$$

The proof of Theorem~\ref{main} is complete.
\end{proof}

\begin{rem}\label{Rem83}
\normalfont
The main difference between the Dirichlet and the Neumann boundary cases is in the logarithmic loss in the local energy decay estimate
\eqref{eq.LE}. Due to this term, comparing our result with the one in \cite{KP}, we see that the estimates for $S_2[\Xi]$ and $S_4[\Xi]$ are worse
in the Neumann case.
\end{rem}

%%%%%%%%%%%%%%%%%%
\renewcommand{\theequation}{A.\arabic{equation}}

\setcounter{equation}{0}  % reset counter
\renewcommand{\thelemma}{A.\arabic{lemma}}

\setcounter{lemma}{0}
\renewcommand{\thetheorem}{A.\arabic{theorem}}

\setcounter{theorem}{0}
%%%%%%%%%%%%%%%%%%%%%%%%%%%%%%%%%%%%%%%
\section*{Appendix: A local existence theorem of smooth solutions}
Here we sketch a proof of the following local existence theorem for the semilinear case
(for the general case, see \cite{SN89}).
We underline that the convexity assumption for the obstacle is not necessary for the local existence result.

\begin{theorem}\label{LE}
Let ${\mathcal O}$ be a bounded obstacle with $\mathcal C^\infty$ boundary and $\Omega=\R^2\setminus \mathcal O$.
For any $\phi$, $\psi\in {\mathcal C}^\infty_0(\overline{\Omega})$
satisfying the compatibility condition of infinite order and
\begin{equation}
\label{Ori-init}
\|\phi\|_{H^{5}(\Omega)}+\|\psi\|_{H^{4}(\Omega)}\le R,
\end{equation}
there exists a positive constant $T=T(R)$ such that
the mixed problem \eqref{eq.PMN} admits a
unique solution $u\in C^\infty\bigl([0,T)\times \overline{\Omega}\bigr)$.
Here $T$ is a constant depending only on $R$.
\end{theorem}

For nonnegative integer $s$, we put
$$
Y^s_T:=\bigcap_{j=0}^{s} {\mathcal C}^{j}\bigl([0,T]; H^{s-j}(\Omega)\bigr),
$$
and
$$
\|h\|_{Y^s_T}:=\sum_{j=0}^s \sup_{t\in [0,T]} \|\pa_t^j h(t,\cdot)\|_{H^{s-j}(\Omega)}.
$$
Let $v_j$ for $j\ge 0$ be given as in Definition~\ref{CCN}.
First we show the following result.

\begin{lemma}\label{LEw}
Let $m\ge 2$.
Suppose that $(\phi$, $\psi)\in H^{m+2}(\Omega)\times H^{m+1}(\Omega)$
satisfies the compatibility condition of order $m+1$, that is to say,
$\left.\pa_\nu v_j\right|_{\pa\Omega}=0$ for $j\in \{0,1,\ldots, m+1\}$,
and
\begin{equation}
\|\phi\|_{H^{m+2}(\Omega)}+\|\psi\|_{H^{m+1}(\Omega)}\le M.
\label{init}
\end{equation}
Then\footnote{The assumption on initial data here is just for simplicity, and
we can prove the same result for initial data with compatibility condition of order $m$ in fact.},
there exists a positive constant $T=T(m, M)$ such that
the mixed problem \eqref{eq.PMN} admits a
unique solution $u\in Y_{T}^{m+2}$. Here $T$ is a constant depending only on $m$ and $M$.
\end{lemma}
\begin{proof} %[Outline of proof]
To begin with, we note that the Sobolev embedding theorem implies
\begin{equation}
\label{Sob01}
\sum_{|\beta|\le [(m+1)/2]+1} \|\pa^\beta h(t,\cdot)\|_{L^\infty_\Omega}
\lesssim \sum_{|\beta|\le [(m+1)/2]+3}\|\pa^\beta h(t,\cdot)\|_{L^2_\Omega}
\le \sum_{|\beta|\le m+2}\|\pa^\beta h(t,\cdot)\|_{L^2_\Omega}
\end{equation}
for $m\ge 2$.

We show the existence of $u$ by constructing an approximate sequence $\bigl\{u^{(n)}\bigr\}\subset Y_T^{m+2}$,
and proving its convergence for suitably small $T>0$.
Throughout this proof, $C_M$ denotes a positive constant depending on $M$, but being independent of $T$.
In order to keep the compatibility condition, we need to choose an appropriate function for the first step:
for a moment, we suppose that %, fixed $T>0$, 
we can choose
a function $u^{(0)}\in Y_T^{m+2}$ satisfying $(\pa_t^j u^{(0)})(0,x)=v_j$ for all $j\in \{0,1,\ldots, m+2\}$.
For $n\ge 1$ we inductively define $u^{(n)}$ as
\begin{equation}
u^{(n)}=S\bigl[\phi, \psi, G\bigl(\pa u^{(n-1)}\bigr) \bigr].
\end{equation}
We have to check that $u^{(n)}$ is well defined.
Let $v_0^{(n)}:=\phi$, $v_1^{(n)}:=\psi$, and $v_j^{(n)}:=\Delta v_{j-2}^{(n)}
+\pa_t^{j-2} (G(\pa u^{(n-1)})\bigr|_{t=0}$
for $j\ge 2$.
Suppose that $u^{(n-1)}\in Y_T^{m+2}$ with $(\pa_t^j u^{(n-1)})(0)=v_j$ for $0\le j\le m+2$.
Then we can see that $v_j^{(n)}=v_j$ for $0\le j\le m+2$, and
consequently the compatibility condition of order $m+1$ is satisfied for the equation
of $u^{(n)}$. Since \eqref{Sob01} implies $G(\pa u^{(n-1)})\in Y_T^{m+1}$, the linear theory (see \cite{I68}) shows that $u^{(n)}\in Y_T^{m+2}$.
Therefore, by induction with respect to $n$, we see that $\{u^{(n)}\}\subset Y_T^{m+2}$ is well defined,
and that $(\pa_t^j u^{(n)})(0)=v_j^{(n)}=v_j$ for $0\le j\le m+2$ and $n\ge 0$.

Now we are going to explain how to construct $u^{(0)}$.
We can show that $v_j\in H^{m+2-j}(\Omega)$ for $0\le j\le m+2$ by its definition and \eqref{Sob01}. By the well-known extension theorem, there is $V_j\in H^{m+2-j}(\R^2)$
such that $\left. V_j\right|_\Omega=v_j$ and
$\|V_j\|_{H^{m+2-j}(\R^2)}\lesssim \|v_j\|_{H^{m+2-j}(\Omega)}$.
Let $(a_{kl})_{0\le k, l\le m+2}$ be the inverse matrix of $(i^k(l+1)^k)_{0\le k, l\le m+2}$, where $i=\sqrt{-1}$.
We put
$$
\widehat{V}
(t, \xi)=\sum_{k,l=0}^{m+2} \exp(i(k+1)\jb{\xi}t)a_{kl}\widehat{V_l}(\xi)\jb{\xi}^{-l},
$$
where $\widehat{V_l}$ is the Fourier transform of $V_l$.
 We set $u^{(0)}(t)=\left. V(t) \right|_{\Omega}$ with the inverse Fourier transform
$V(t)$ of $\widehat{V}(t)$.
Now we can show that $u^{(0)}(t)$
has the desired property, and
$\|u^{(0)}\|_{Y_T^{m+2}}\le C_M$ (see \cite{SN89} where this kind of function
is used to reduce the problem to the case of zero-data).

Now we are in a position to show that $u^{(n)}$ converges  to a local solution of \eqref{eq.PMN} on $[0,T]$ with appropriately chosen $T$.
For simplicity of description, we put
$$
\Norm{h(t)}_k=\sum_{j=0}^{m+2-k}\|\pa_t^j h(t)\|_{H^k(\Omega)}
$$
for $0\le k\le m+2$. Note that we have $\|h\|_{Y_T^{m+2}}\lesssim \sup_{t\in[0,T]}\sum_{k=0}^{m+2}\Norm{h(t)}_k$.
We also set $G_n(t,x)=G\bigl(\pa u^{(n)}(t,x)\bigr)$ for $n\ge 0$.
Combining the elementary inequality
$$
\|h(t)\|_{L^2_\Omega}\le \|h(0)\|_{L^2_\Omega}+\int_0^t \|(\pa_t h)(\tau)\|_{L^2_\Omega}d\tau
$$
with the standard energy inequality for $\pa_t^j u^{(n)}$ with $0\le j\le m+1$, we get
$$
\Norm{u^{(n)}(t)}_0+\Norm{u^{(n)}(t)}_1\le (1+T)\left(C_M+C\sum_{j=0}^{m+1}\int_0^t \|(\pa_t^j G_{n-1})(\tau)\|_{L^2_\Omega}d\tau\right).
$$
Writing
$$
\Delta \pa^\beta u^{(n)}(t,x)=\pa_t^2 \pa^\beta u^{(n)}-(\pa^\beta G_{n-1})(0,x)
{}-\int_0^t (\pa_t\pa^\beta G_{n-1})(\tau,x)d\tau
$$
for a multi-index $\beta$
and using the elliptic estimate, given in Lemma \ref{elliptic2}, we have
$$
\Norm{u^{(n)}(t)}_k\le C \left(\Norm{u^{(n)}(t)}_{k-2}+\Norm{u^{(n)}(t)}_{k-1}+C_M+\sum_{|\alpha|\le k-1}\int_0^t \|(\pa^\alpha G_{n-1})(\tau)\|_{L^2_\Omega}d\tau\right)
$$
for $2\le k\le m+2$.
By induction we get control of $\Norm{u^{(n)}(t)}_k$ for $0\le k\le m+2$, and obtain
\begin{equation} 
\label{EnergyLocal}
\sum_{k=0}^{m+2} \Norm{u^{(n)}(t)}_k\le (1+T)\left(C_M+C\sum_{|\alpha|\le m+1}\int_0^t
\|(\pa^\alpha G_{n-1})(\tau)\|_{L^2_\Omega}d\tau\right).
\end{equation}
It follows from \eqref{Sob01} that
\begin{equation}
\label{EstNonlinearity01}
 \sum_{|\alpha|\le m+1}
 \|(\pa^\alpha G_{n-1})(\tau)\|_{L^2_\Omega}
  \le C\|u^{(n-1)}\|_{Y^{m+2}_T}^3, \quad 0\le \tau\le T,
\end{equation}
and \eqref{EnergyLocal} implies $\|u^{(n)}\|_{Y_T^{m+2}}\le (1+T)\left(C_M+CT\|u^{(n-1)}\|_{Y_T^{m+2}}^3\right)$
for $n\ge 1$. From this,
if we take appropriate constants $N_M$ and $T_M$ which can be determined by $M$,
we can show that $\|u^{(n)}\|_{Y_T^{m+2}}\le N_M$ for all $n\ge 0$, provided that $T\le T_M$.
In the same manner, we can also show that there is some $T_M'(\le T_M)$ such that
$$
\|u^{(n+1)}-u^{(n)}\|_{Y_T^{m+2}}\le \frac{1}{2}\|u^{(n)}-u^{(n-1)}\|_{Y_T^{m+2}}
$$
for all $n\ge 1$, provided that $T\le T_M'$. Now we see that if $T\le T_M'$, then $\{u^{(n)}\}$ is a Cauchy sequence in $Y_T^{m+2}$, and there is $u\in Y_T^{m+2}$ such that $\lim_{n\to\infty}\|u^{(n)}-u\|_{Y_T^{m+2}}=0$.
It is not difficult to see that this $u$ is the desired solution to \eqref{eq.PMN}.

Uniqueness can be easily obtained by the energy inequality.
\end{proof}
%%%%%%%%%%%%%%%%%%%%%%%%%
Theorem~\ref{LE} is a corollary of Lemma~\ref{LEw}.
%%%%%%%%%%%%%%%%%%%%%%%%%
\begin{proof}[
Proof of Theorem~\ref{LE}]
The assumption on the initial data guarantees that
for each $m\ge 3$, there is a positive constant $M_m$ such that $\|\phi\|_{H^{m+2}(\Omega)}+\|\psi\|_{H^{m+1}(\Omega)}\le M_m$.
Hence, by Lemma~\ref{LEw},
there is $T_m=T(m, M_m)>0$ such that \eqref{eq.PMN} admits a unique solution $u\in Y_{T_m}^{m+2}$.
Note that we may take $T_3=T(3,R)$. We put
\begin{equation}\label{eq.bi}
C_0:=\|u\|_{Y_{T_3}^{3+2}}.
\end{equation}
Our aim is to prove that
\eqref{eq.PMN} admits a solution $u\in \bigcap_{m\ge 3}Y_{T_3}^{m+2}$.
Then the Sobolev embedding theorem implies that $u\in {\mathcal C}^\infty\left([0,T_3]\times \overline{\Omega}\right)$,
which is the desired result. For this purpose,
we are going to prove the following {\it a priori} estimate:
for each $m\ge 3$, if $u\in Y^{m+2}_T$ is a solution to  \eqref{eq.PMN} with some $T\in (0, T_3]$,
then there is a positive constant $C_m$, which is independent of $T$, such that
\begin{equation}
\label{LW_conclusion}
 \|u(t)\|_{Y_T^{m+2}}\le C_m.
\end{equation}
Once we obtain this estimate, by applying Lemma~\ref{LEw} repeatedly, we can see that $u\in Y_{T_3}^{m+2}$
for each $m\ge 3$, which concludes the proof of Theorem~\ref{LE}.

Now we show \eqref{LW_conclusion} by induction. For $m=3$ \eqref{LW_conclusion} follows immediately from \eqref{eq.bi}.
Suppose that
we have \eqref{LW_conclusion} for some $m=l\ge 3$.
If we put
$$
\Norm{h(t)}_k=\sum_{j=0}^{l+3-k}\|\pa_t^j h(t)\|_{H^k(\Omega)},
$$
then, similarly to \eqref{EnergyLocal}, we obtain
$$
\sum_{k=0}^{l+3} \Norm{u(t)}_k\le (1+T_3)\left(C+C\sum_{|\alpha|\le l+2}\int_0^t
\bigl\|(\pa^\alpha \bigl(G\bigl(\pa u(\tau)\bigr)\bigr)\bigr\|_{L^2_\Omega}d\tau\right).
$$
Since $[(m+1)/2]+3\le m+1$ for $m\ge 4$, we have
\begin{equation}
\label{Sob02}
\sum_{|\beta|\le [(m+1)/2]+1} \|\pa^\beta h(t,\cdot)\|_{L^\infty_\Omega}
\le C \sum_{|\beta|\le m+1}\|\pa^\beta h(t,\cdot)\|_{L^2_\Omega},\quad  m\ge 4,
\end{equation}
in place of \eqref{Sob01}. Combining this estimate for $m=l+1$ with the inductive assumption, we get
$$
\sum_{|\alpha|\le l+2}\bigl\|\pa^\alpha\bigl(G(\pa u(\tau))\bigr)\bigr\|_{L^2_\Omega}\le CC_l^2
\sum_{k=0}^{l+3}\Norm{u(\tau)}_k,
$$
which yields
$$
\sum_{k=0}^{l+3} \Norm{u(t)}_k\le (1+T_3)\left(C+CC_l^2\int_0^t \sum_{k=0}^{l+3} \Norm{u(\tau)}_kd\tau\right).
$$
Now the Gronwall Lemma implies $\sum_{k=0}^{l+3} \Norm{u(t)}_k\le C(1+T_3)\exp\bigl(C C_l^2(1+T_3)T_3\bigr)=:C_{l+1}$
for $0\le t\le T(\le T_3)$, which implies $\|u\|_{Y_T^{l+3}}\le C_{l+1}$ for $0\le T\le T_3$.
This completes the proof of \eqref{LW_conclusion}.
\end{proof}

\begin{center}
{\bf Acknowledgments}
\end{center}
The first author is partially supported by Grant-in-Aid for Scientific Research (C) (No. 23540241), JSPS.
The second author is partially supported by Grant-in-Aid for Science Research (B) (No. 24340024), JSPS.
The third author is partially supported by GNAMPA Projects 2010 and 2011, coordinating
by Prof. P. D'Ancona.


\begin{thebibliography}{abc99}
\bibitem[ADN59]{ADN}
{\sc S. Agmon, A. Douglis, and L. Nirenberg},
{\it Estimates near the boundary for solutions of elliptic
partial differential equations satisfying general boundary conditions I},
Comm. Pure Appl. Math.
{\bf 12} (1959), 623--737.

\bibitem[D03]{DiF03}
{\sc M. Di Flaviano},
{\it Lower bounds of the life span of classical solutions to a system
of semilinear wave equations in two space dimensions},
 J. Math. Anal. Appl.
{\bf 281} (2003), 22--45.

\bibitem[GL04]{GeLu04}
{\sc V. Georgiev and S. Lucente},
{\it Decay for Nonlinear Klein-Gordon Equations},
NoDEA
{\bf 11} (2004), 529--555.

\bibitem[G93]{God93} {\sc P. Godin},
{\it Lifespan of solutions of semilinear wave equations in two space
dimensions},
Comm. Partial Differential Equations, {\bf 18} (1993), 895--916.

\bibitem[I68]{I68}
{\sc M. Ikawa},
{\it Mixed problems for hyperbolic equations of second order},
J. Math. Soc. Japan
{\bf 20} (1968), 580--608.

\bibitem[KK08]{KaKu08}
{\sc S. Katayama and H. Kubo},
{\it An elementary proof of global existence for nonlinear wave equations in an exterior domain},
J. Math. Soc. Japan
{\bf 60} (2008),  1135--1170.
\bibitem[KK12]{KaKu12}
{\sc S. Katayama and H.Kubo},
{\it Lower bound of the lifespan of solutions to semilinear wave equations in an exterior domain},
ArXiv: 1009.1188.

\bibitem[Kl85]{kl0}
{\sc S. Klainerman},
{\it Uniform decay estimates and the Lorentz invariance of the classical wave equation},
 Comm. Pure Appl. Math. {\bf 38} (1985), 321--332.

\bibitem[K07]{Kub06}
{\sc H. Kubo},
{\it Uniform decay estimates for the wave equation in an exterior domain},
\lq\lq Asymptotic analysis and singularities\rq\rq, pp.~31--54, Advanced Studies in Pure Mathematics
47-1, Math. Soc. of Japan, 2007.

\bibitem[K12]{KP}
{\sc H. Kubo},
{\it Global existence for nonlinear wave equations in an exterior domain in 2D},
ArXiv: 1204.3725v2.

\bibitem[Ku93]{k93}
{\sc K. Kubota},
{\it Existence of a global solutions to a semi-linear wave equation
with initial data of non-compact support in low space dimensions},
 Hokkaido Math. J. {\bf 22} (1993), 123--180.


\bibitem[M75]{Mor75}
{\sc C. S.~Morawetz},
{\it Decay for solutions of the exterior problem for the wave equation},
Comm. Pure Appl. Math. {\bf 28} (1975), 229--264.

\bibitem[SS03]{SeSh03}
{\sc P. Secchi and Y. Shibata},
{\it On the decay of solutions to the 2D Neumann exterior problem for the wave equation},
J. Differential Equations
{\bf 194} (2003), 221--236.

\bibitem[SN89]{SN89}
{\sc Y. Shibata and G. Nakamura},
{\it On a local existence theorem of Neumann problem for some quasilinear hyperbolic systems of 2nd order },
Math. Z,
{\bf 202} (1989), 1--64.

\bibitem[SSW11]{SSW}
{\sc H.F. Smith, C.D. Sogge, and C. Wang},
{\it Strichartz Estimates for Dirichlet-Wave Equations in Two Dimensions with Applications},
Transactions Amer. Math. Soc. \bf 364 \rm (2012), 3329-3347.

\bibitem[V75]{Vai75}
{\sc B.R. Vainberg},
{\it
The short-wave asymptotic behavior of the solutions of stationary problems,
and the asymptotic behavior as $t\rightarrow \infty $ of the solutions of
nonstationary problems},
(Russian) Uspehi Mat. Nauk
{\bf 30} (1975), 3-55.

\end{thebibliography}
\end{document}